\documentclass[10pt]{amsart}
\let\Q\QQ
\newcommand{\Q}{\mathbb Q}
\newcommand{\Top}{\mathcal S}

\newcommand{\Sqp}{\Sq^2_+}
\usepackage{mathabx,epsfig}

\input{corank2.sty}

\title[Enumerating corank 2 complex vector bundles on odd complex projective spaces]{Enumerating corank 2 complex vector bundles on odd complex projective spaces}

\author{Yang Hu}
\address[Y. Hu]{University of Regina}
\email{yang.hu@uregina.ca}

\author{Morgan Opie}
\address[M. Opie]{Northwestern University}
\email{mpopie@northwestern.edu}

\date{}                     

\begin{document}

\begin{abstract}
For $n$ odd, we enumerate rank $n$ complex topological vector bundles on $\mathbb CP^{n+2}$ with prescribed complex $K$-theory class. The answer exhibits a 24-fold periodicity and is closely related to the enumeration of corank $2$ stably trivial vector bundles. 
\end{abstract}
\maketitle
\tableofcontents

\section{Introduction}


Given a finite cell complex $X$, a classical problem is to understand all fixed-rank complex topological vector bundles over $X$. To do so, we begin with complex $K$-theory, the Grothendieck group of finite-rank complex topological vector bundles on $X$ under direct sum. Complex $K$-theory is a rich, well-understood, and often computable invariant by work of Bott \cite{Bott}, Atiyah--Hirzebruch \cite{AHSSeq}, and others. We then ask: 
\begin{q}\label{q:1} Let $h$ be a reduced complex $K$-theory class on $X$. Does $h$ admit a rank $n$ representative? If so, how many isomorphism classes of such representatives are there?
\end{q}
In this paper we answer \Cref{q:1} for $X=\CP^{n+2}$, with $n$ odd and $h$ arbitrary. Since the $K$-theory of complex projective spaces is well-known, this completely describes rank $n$ complex topological vector bundles on $\CP^{n+2}$ when $n$ is odd.
To state our main theorem, we introduce some notation. Given a reduced $K$-theory class $h$ on $X$ as in \Cref{q:1}, we write
\[
\Vect_n^h(X)
\]
for the set of isomorphism classes of rank $n$ complex vector bundles on $X$ representing $h$ in $K$-theory.

\begin{thmx}[{\Cref{main3count}, \Cref{thm:n3mod4}, \Cref{cor:1mod4count}, \Cref{thm:n5mod8-2local}}]\label{thm:main} Let $n\geq 3$ be odd and let $h$ be a reduced complex $K$-theory class on $\CP^{n+2}$. Then $\Vect_n^h(\CP^{n+2})$ is non-empty if and only if the $(n+1)$-st and $(n+2)$-nd Chern classes of $h$ vanish. When this set is non-empty, the size of $\Vect^h_{n}(\CP^{n+2})$ is given by a
$24$-periodic formula depending only on $n$ modulo $24$ and the first and second Chern classes of $h$ modulo $12$, according to the formulas in \Cref{table:1}.

\begin{figure}[h]
\centering
\textbf{The size of $\Vect^h_n(\CP^{n+2})$ in terms of $n$, $c_1(h)$, and $c_2(h)$.}\par\medskip
{\setlength{\arraycolsep}{5pt}
\renewcommand{\arraystretch}{1.2}
\[
\begin{array}{c@{\qquad}cccccc}
\hline
n \bmod 24 & 1 & 3 & 5 & 7 & 9 & 11 \\
\noalign{\hrule height 0.8pt}
\#\Vect^h_n & (2,c_2) & (3,c_1,c_2) & (4,2c_1+c_2) & 1 & (6,2c_1,c_2) & 1 \\ \\
\hline 
n \bmod 24 & 13 & 15 & 17 & 19 & 21 & 23 \\
\noalign{\hrule height 0.8pt}
\#\Vect^h_n & (4,2c_1+c_2) & (3,c_1,c_2) & (2,c_2) & 1 & (12,4c_1,2c_1+c_2) & 1 \\
\end{array}
\]
}
\caption{In the above, $c_1=c_1(h)$ and $c_2=c_2(h)$ are the first and second Chern classes of $h$, respectively; we identify these classes with integers for complex projective spaces. $\#\Vect^h_n$ is the size of the set $\Vect^h_n(\CP^{n+2})$. For integers $a,b,c$, the the symbols $(a,b)$ and $(a,b,c)$ denote the corresponding greatest common divisors.}\label{table:1}
\end{figure}

\end{thmx}
To explain the homotopy-theoretic approach to \Cref{q:1}, we recall the objects representing vector bundles and $K$-theory. Let $U(n)$ denote the unitary group of rank $n$, and $U=U(\infty)$ be the stabilized unitary group. The classifying spaces $BU(n)$ and $BU$ represent rank $n$ complex vector bundles and reduced complex $K$-theory classes, respectively. The standard inclusion $U(n) \hookrightarrow U$ induces a natural map $BU(n)\to BU$ representing the stabilization of rank $n$ complex vector bundles, taking such a bundle to its $K$-theory class. Thus, answering \Cref{q:1} amounts to studying homotopy classes of lifts of $h$ along the natural map $BU(n)\to BU$, as depicted in \Cref{diagram1}. 
\begin{equation}\label[diagram]{diagram1}
\begin{tikzcd}
& BU(n)\ar[d]\\
X \ar[r,"h"] \ar[ur,dashed] & BU.
\end{tikzcd}
\end{equation}
We use {\em Moore--Postnikov obstruction theory} to treat this lifting problem. 

Fixed-rank complex vector bundles on projective spaces have been extensively
studied using obstruction theory and related unstable methods; see,
for example, \cite{AR,Switzer,Switzer2,Opie-r3p5}. Our results continue
this tradition.
Two previous calculations provide particularly useful points of comparison. First, Theorems 1.1 and 1.2 of \cite{Hu} provide answers to the stably trivial enumeration problems in coranks $1$ and $2$. Thus, we already have a complete understanding of $\Vect_n^0(\CP^{n+1})$ and $\Vect_n^0(\CP^{n+2})$. Second, nonzero stable classes in corank $1$ are completely understood by \cite[Theorem 1.3]{opie24enum}: the set $\Vect_n^h(\CP^{n+1})$ has a simple description for any choice of $h$, depending only on the parity of $n$ and $c_1(h)$, and whether or not $c_{n+1}(h)$ vanishes. Given these results, a natural next step is to consider corank $2$ bundles on $\CP^{n+2}$ for $h\neq 0$, which \Cref{thm:main} accomplishes when $n$ is odd. 

We note that \Cref{thm:main} shows that the cardinality of $\Vect^h_n(\CP^{n+2})$ is periodic in $n$ and is controlled by divisibility properties of $c_1(h)$ and $c_2(h)$. It also follows from the displayed formulas that $\#\Vect_n^h(\CP^{n+2})$ divides the corresponding cardinality for $h=0$. These structural features are also present in the earlier corank $1$ calculation for arbitrary $h$ \cite[Theorem 1.3]{opie24enum}. Thus, the authors view \Cref{thm:main} as computational evidence of structure in $\Vect_r^h(\CP^m)$, at least for $m-r$ not too large. More explicitly, future exploration may be guided by the following:
\begin{q}\label{conj:bound} Fix a dimension $m$ and a rank $r$ in the {\em metastable range} (that is, for $\frac{m}{2}\leq r<m$). Let $h\: \CP^{m} \to BU$ be a reduced $K$-theory class. Are Chern classes the only obstruction to $h$ admitting a rank $r$ representative? Is the size of  $\Vect^h_r(\CP^m)$ always a divisor of the size of $\Vect^0_r(\CP^m)$? Is there a closed formula for the size of $\Vect^h_r(\CP^m)$ in terms of $m$, $r$, and small Chern classes of $h$? \end{q} 

\begin{rmk} \Cref{conj:bound} makes sense for any $m$ and $r$, but the metastability restriction seems likely for any reasonable uniform answer. Structural results for stably trivial vector bundles only hold in this range (see \cite[Theorem 2.1]{Hu}). 
Moreover, outside the metastable range, Chern classes need not be the
only obstructions to the existence of a fixed-rank representative. By
\cite[Theorem 3(c)]{Switzer2}, there are $K$-theory classes $h$ on
$\CP^5$ for which $c_i(h)=0$ for all $i\geq 3$, but which do not admit
rank $2$ representatives. This provides evidence for additional
unstable phenomena in low rank. \end{rmk} 

\subsection{Outline}

In \Cref{prelim}, we collect relevant background. We begin in \Cref{weiss} with an overview of Weiss' unitary calculus, to motivate the stably trivial story. This is followed by a discussion of classical fracture theorems in unstable homotopy (\Cref{fracture}) and some setup for lifting problems (\Cref{lifting}).

\Cref{structure} contains key structural observations, which substantially simplify the computations in the rest of the paper.

The remainder of the paper provides computational ingredients for \Cref{thm:main}. 
In \Cref{MoorePost3}, we give a $3$-primary Moore--Postnikov argument that allows us to completely analyze $3$-divisibility of $\#\Vect^h_n(\CP^{n+2})$. The challenging computations are in \Cref{MoorePost2}, where we give separate $2$-primary arguments for each odd residue class of $n$ modulo $8$, and appeal to background theory from \Cref{structure} to simplify arguments as much as possible. By doing so, we only need to analyze two cases: $n \equiv 1,5\pmod{8}$. Of these, the case $n\equiv 1\pmod{8}$ is more straightforward, while $n\equiv 5\pmod{8}$ requires more extensive analysis. 

\Cref{app:Z4-cohomology} includes some $\Z/4$-cohomology computations needed for the case $n\equiv 5\pmod{8}$, while \Cref{app:Chernformulas} records some standard results on Steenrod operations on Chern classes.


\subsection{Notations and conventions}\label{sec:not}
Throughout this paper, we use the following notations and conventions.
\begin{itemize}
\item All spaces have the homotopy type of a CW complex.
\item Given spaces $X$ and $Y$, we write $[X,Y]$ for homotopy classes of maps from $X$ to $Y$.  We write $Y^X$ for the space of maps from $X$ to $Y$, so that $[X,Y]=\pi_0(Y^X)$. 
\item Given a finite simply connected CW  complex $X$, let $h \in [X,BU]$ be a fixed (reduced) complex $K$-theory class. We let $\Vect_n^h(X)$ denote the set of isomorphism classes of rank $n$ complex vector bundles on $X$ representing the $K$-theory class $h$. If $h=0$, $\Vect^h_n(X)$ is the set of isomorphism classes of stably trivial rank $n$ vector bundles on $X$.
\item Given a pointed space $X$, we write $\pi_iX$ for its $i$-th unstable homotopy group.  Given any group or ring $G$, we usually write $H^*_{G} X$ for the cohomology of $X$ with $G$ coefficients. In some situations, e.g., if $G$ has a long name, we write $H^*(X,G)$.
\item Let $\Top$ denote the category of topological spaces. The $m$-truncation functor $\tau_{\leq m}\: \Top \to \Top_{\leq m}$ is the reflection onto the full subcategory $\Top_{\leq m}$ of spaces $X$ with $\pi_iX=0$ for $i>m$, i.e., the left adjoint to the inclusion $\Top_{\leq m} \hookrightarrow \Top.$
\item A space $Y$ is {\em $m$-skeletal} if, for all $X$, $[Y,X]\xrightarrow{\cong} [Y,\tau_{\leq m}X]$ via composition with $m$-truncation. CW complexes with cells of dimension at most $m$ are $m$-skeletal.

\item Given a prime $p$ and a group, ring, space, or spectrum $X$, we write $\c{X}{p}$ for $p$-completion.
\item Given a graded module $M$ and an integer $k$, we write $M^{*\leq k}$ for the quotient of $M$ by the submodule of elements of degree greater than $k$.
\item We write $\Sp$ for the category of spectra.
\item Given an Eilenberg--Mac Lane space $K(G,n)$, we write $\iota_n$ for a generator of $H^n(K(G,n),G)$.
\end{itemize}

\subsection{Acknowledgements} The authors are grateful to Ben Antieau, Paul Goerss, Gijs Heuts, Mike Hill, Alexander Smith, Niall Taggart, and Allen Yuan for useful conversations.
The first-named author is supported by the Pacific Institute of Mathematical Sciences (PIMS) through CRG41.
The second-named author was partially supported by a National Science Foundation Mathematical Sciences Postdoctoral Research Fellowship, No. 2202914, at the beginning of this project.

\section{Preliminaries}\label{prelim}
In this section, we recall some background ideas. In \Cref{weiss}, we discuss enumeration of stably trivial bundles via {\em Weiss calculus} (also known as {\em unitary calculus}) as an organizing principle and source of inspiration for the general case. In \Cref{fracture} and \Cref{lifting}, we briefly review some classical homotopy theory underpinning the more technical computations in the body of this paper. 
\subsection{Weiss calculus and vector bundles}\label{weiss}

Weiss introduced a functor calculus in \cite{Weiss}, which includes various flavours: {\em orthogonal calculus} (the focus of Weiss' original paper), {\em unitary calculus} (see, e.g., \cite{Arone02}, \cite{Hu}, \cite{taggart22unitary}), and {\em symplectic calculus} (as introduced in \cite{carrtagg24}). The general set-up in any variant of Weiss calculus is to approximate functors from some topological category of vector spaces to pointed topological spaces by {\em polynomial functors}. Prototypical examples of functors where this formalism can be fruitfully applied include the functor $BO(-)$ with domain real inner product spaces, and the functor $BU(-)$ with domain complex inner product spaces. In both these cases, the resulting {\em Weiss tower} of polynomial approximation functors converges to the original functor. In general, the fibers of maps between subsequent layers in the Weiss tower evaluate to infinite loop spaces. 

A key idea in \cite{Hu} is to use the (unitary) Weiss tower for $BU(-)$ to enumerate complex vector bundles. Evaluating this Weiss tower at an $r$-dimensional complex inner product space results in a tower of fibrations 
\[\to T_iBU(r) \to T_{i-1}BU(r) \to \ldots \to T_1BU(r) \to BU,\] 
where the fiber of $T_iBU(r) \to T_{i-1}BU(r)$ is an infinite loop space and arbitrarily highly connected as $i$ goes to infinity. In the so-called {\em metastable range}, i.e., for spaces of CW dimension up to $2m$ where 
$\frac{m}{2} \leq r <m$, these connectivity estimates reduce the study of rank $r$ complex vector bundles to the study of maps into $T_1BU(r)$. The fiber  $L_1BU(r)$ of $T_1BU(r) \to BU$ is described abstractly by \cite[Theorem 2]{Arone02} and explicitly in \cite[Proposition 2.3]{Hu}: 
\[L_1BU(r) \cong \Omega^\infty \Sigma \CP^\infty_r,\] 
where the spectrum $\CP^\infty_r$ is the cofiber of the natural map $\Sigma^\infty_+\CP^{r-1} \to \Sigma^\infty_+ \CP^\infty$.
For rank $r$ vector bundles on $\CP^m$ with $\frac{m}{2} \leq r < m$, Hu deduces the following \cite[Theorem 2.1]{Hu}:
\[\Vect^0_r(\CP^m) \cong [ \CP^m,\Omega^\infty\Sigma \CP^{\infty}_r] \cong \pi_0\map_{\Sp}(\CP^m_r,\Sigma \CP^{m}_r).\]
Thus, stably trivial vector bundles are naturally identified with the zeroth homotopy group of a mapping spectrum between finite cell complexes. The first-named author uses this description to completely compute corank $1$ and corank $2$ stably trivial complex topological vector bundles on complex projective spaces \cite[Theorems 1.1, 1.2]{Hu}.  The corank $2$ result is below.
\begin{thm}[Hu '23, {\cite[Theorem 1.2]{Hu}}]\label{thm:yang-main} The size of $\Vect^0_n(\CP^{n+2})$ is as follows:
\[
\begin{tikzpicture}
\matrix (m) [matrix of math nodes,
nodes in empty cells,nodes={minimum width=2ex,
minimum height=2ex,outer sep=-1pt},
column sep=.2ex,row sep=.1ex]{
n \mod{24} & 0 & 1 & 2 & 3 & 4 & 5 & 6 & 7 & 8 & 9 & 10 & 11 & \, \\
\#\Vect^0_n(\CP^{n+2}) &  12 & 2 & 1&3&2 &4&3& 1 & 4& 6 & 1 & 1&\,
\\ \\ 
  n \mod{24} & 12 & 13 & 14 & 15 &16 &17  & 18 & 19 & 20 &21 & 22 & 23& \, \\
\#\Vect^0_n(\CP^{n+2})  & 6& 4 & 1 & 3 & 4 & 2& 3 & 1 & 2& 12 & 1 & 1 \strut \\};
\draw[thick] (m-1-1.north west) -- (m-1-13.north east);
\draw[] (m-1-1.south west) -- (m-1-13.south east);
\draw[thick] (m-4-1.north west) -- (m-4-13.north east);
\draw[] (m-4-1.south west) -- (m-4-13.south east);
\end{tikzpicture}
\]
\end{thm}

\subsection{Fracturing problems}\label{fracture}
The fracture square in homotopy theory is analogous to the Hasse principle in number theory (namely, checking an arithmetic property over $\bb{Z}$ is equivalent to checking it over $\bb{Q}$ and over the $p$-adic integers $\bb{Z}_p$ for every prime $p$). 
Given a nilpotent space $Z$, its arithmetic completion at a prime $p$, denoted $\c{Z}{p}$, or rationalization, denoted $Z_{\mathbb Q}$, is constructed and discussed in detail in \cite{MP}. The key result we rely on here is the following:

\begin{thm}[Fracture Theorem for Completion, {\cite[Theorem 13.5.3]{MP}}]
   \label{thm:fracture-spaces}
   The following is a homotopy pullback square for every finite-type, simple space $Z$.
\begin{equation} \label[diagram]{fracture-square}
      \begin{tikzcd}[column sep = large]
           Z \ar[r] \ar[d] & \prod_p \c{Z}{p} \ar[d] \\
           Z_{\bb{Q}} \ar[r] & (\prod_p \c{Z}{p})_{\bb{Q}}.
      \end{tikzcd}
\end{equation}
\end{thm}
Combining \Cref{thm:fracture-spaces} above with \cite[Theorem 11.1.1]{MP}, the universal coefficient theorem, and the Hurewicz theorem, we note the following standard result:
\begin{lemma} \label{lem:basic-0}  Let $X$ and $Y$ be simply connected, finite-type spaces. If $g\:X \to Y$ induces an isomorphism on cohomology with coefficients in $\Q$ and $\Z/p$ for all $p$, then $g$ is a homotopy equivalence. If $g$ induces a cohomology isomorphism in degrees less than or equal to $i$ with coefficients in $\mathbb Q$ and with coefficients $\mathbb F_p$ for all primes $p$, then $g$ induces an isomorphism on homotopy in degrees less than or equal to $i-1$, and a surjection in degree $i$. \end{lemma} 
\Cref{fracture-square} allows us to work prime-by-prime when mapping into a simple, finite-type space:
\begin{cor}\label{cor:pcomplete}
Let $Y$ and $Z$ be spaces, with $Z$ simple and finite-type. Homotopy classes of maps from $Y$ to $Z$ are in bijection with compatible choices of maps from $Y \to Z_{\Q}$ and $Y \to \c{Z}{p}$.
\end{cor}

Applying this to $Z=BU(n)$, which is simply connected and finite-type since its skeleta are Grassmannians, we note that $BU(n)_{\Q}\cong \prod_{i=1}^n K(\Q,2i)$, induced by the total Chern class map. In particular, rationalized vector bundles are uniquely determined by Chern classes.

\begin{cor}\label{cor:product} Fix a map $h\: \CP^{n+2} \to BU$ so that $c_{n+1}(h)=c_{n+2}(h)=0$. Let $\c{\Vect^h_n(\CP^{n+2})}{p}$ denote the set of isomorphism classes of lifts of $\c{h}{p}\:\CP^{n+2} \to \c{BU}{p}$ along the natural map $\c{BU(n)}{p}\to \c{BU}{p}$. Then:
\[\# \Vect^h_n(\CP^{n+2}) =\prod_p \# \c{\Vect^h_n(\CP^{n+2})}{p}.\]
\end{cor}

For $p\geq 5$, it is easy to see that $\c{ {(c_{n+1} \times c_{n+2})}}{p}$ is the only obstruction to lifting a map $\CP^{n+2} \to \c{BU}{p}$ uniquely along $\c{BU(n)}{p} \to \c{BU}{p}$. In particular, we deduce:

\begin{cor}\label{cor:product23} Fix a map $h\: \CP^{n+2} \to BU$ so that $c_{n+1}(h)=c_{n+2}(h)=0$. Let $\c{\Vect^h_n(\CP^{n+2})}{p}$ be as in \Cref{cor:product}. Then:
\[\# \Vect^h_n(\CP^{n+2}) = \# \c{\Vect^h_n(\CP^{n+2})}{2} \cdot \# \c{\Vect^h_n(\CP^{n+2})}{3} \ .\]
\end{cor}
So, we have reduced the problem of computing $\#\Vect^h_n(\CP^{n+2})$ to $2$- and $3$-primary computations.

We will also use the following standard Whitehead theorem for $p$-complete spaces. \begin{lemma}\label{lem:basic} Let $p$ be a prime, let $i\geq 2$, and let $X$ and $Y$ be simply connected, finite-type, $p$-complete spaces. If $g\:X\to Y$ induces an isomorphism on cohomology with coefficients in $\Z/p$ in degrees less than or equal to $i$, then $g$ induces an isomorphism on homotopy groups in degrees less than or equal to $i-1$ and a surjection in degree $i$.\end{lemma}

\subsection{Lifting calculus}\label{lifting}

We describe a classical framework for enumerating lifts along fibrations. Throughout this subsection, all spaces are pointed and all mapping spaces are pointed.

Suppose that we have two fiber sequences \begin{equation}\label[empty]{eq:fib1}P_1 \to P_0 \xrightarrow{w} C, \text{ and }\end{equation} \begin{equation}\label[empty]{eq:fib2}P_0 \xrightarrow{p} B \to K \end{equation} of pointed spaces. Suppose that $X$ is a simply connected space, and that we have a pointed map $g_0\:X\to P_0$ that lifts to $P_1$. We wish to enumerate homotopy classes of lifts of $g_0$ to $P_1$, i.e., fillers of the diagram below up to homotopy:
\begin{equation}\label[diagram]{diagram:lifting-aux}\begin{tikzcd}
& P_1\ar[d]\\
X \ar[r,"g_0"] \ar[ur, dashed]& P_0 \ar[r,"w"] &C
\end{tikzcd}
\end{equation}
 Under suitable hypotheses, the fibration \Cref{eq:fib2} can be used to solve the lifting problem of \Cref{diagram:lifting-aux}. Roughly speaking, this goes as follows: the first fiber sequence \Cref{eq:fib1} identifies the set of lifts with a cokernel of the induced map $\pi_1w_*\: \pi_1(P_0^X,g_0)\to \pi_1(C^X,0)$; the second supplies an action of $\Omega K$ on $P_0$, which induces an action on mapping spaces and often allows the image of $\pi_1w_*$ to be computed by varying an explicit cohomology class. 

To make this strategy more explicit, we begin with some elementary observations about \Cref{diagram:lifting-aux}.

\begin{lemma}\label{lem:lift-0}
Let $g_1\: X \to P_1$ be a lift of $g_0$. Then lifts of $g_0$ up to homotopy are in bijection with
elements in the orbit of $g_1$ under the action of $\pi_1(C^X,0)$ on $\pi_0(P_1^X,g_1)$, which is furthermore identified with the cosets of the image of \[\pi_1w_*:=\pi_1(w\circ -)\: \pi_1(P_0^X,g_0)\to  \pi_1(C^X,0).\]
\end{lemma} 
\begin{proof} We have a fiber sequence
\[  (P_1^X,g_1) \to  (P_0^X,g_0) \to (C^X,0),\]
and so a long exact sequence (of pointed sets at the $\pi_0$-portion):
\[ \pi_1(P_0^X,g_0) \xrightarrow{\pi_1w_*} \pi_1(C^X,0) \to \pi_0(P_1^X,g_1) \to \pi_0(P_0^X,g_0).\]
In particular, the preimage of the basepoint in  $\pi_0(P_0^X,g_0)$ is the set of homotopy classes of all pointed maps $X \to P_1$ that are homotopy lifts of $g_0$. 
This set is precisely the orbit of the basepoint under the natural action of $\pi_1(C^X,0)$ on $\pi_0(P_1^X,g_1)$, which is in bijection with cosets of the stabilizer of $g_1$. By exactness, the stabilizer is the image of $\pi_1w_*$.
\end{proof}

\begin{cor}\label{cor:coker-1} With set-up as in \Cref{lem:lift-0}, if $\pi_1(C^X,0)$ is abelian, then homotopy lifts of $g_0$ to $P_1$ are in bijection with \[\op{coker}\left(\pi_1w_*\:  \pi_1(P_0^X,g_0)\to  \pi_1(C^X,0)\right).\]\end{cor}

We now come to the role of the fiber sequence \Cref{eq:fib2}. The action of loops on the base on the total space arising from \Cref{eq:fib2} gives a map
\begin{equation}\label{def:tildem1}  m\: \Omega K \times  P_0 \to  P_0.\end{equation}
\begin{lemma}\label{lem:lift-22}
Let $g_0\: X \to P_0$ and let $g=p\circ g_0\: X \to B$.
Suppose also that the morphism $\pi_1(P_0^X,g_0) \to \pi_1(B^X,g)$ is zero, and that $\pi_1(C^X,0)$ is abelian. 
Fix any element $y  \in \pi_1(P_0^X,g_0)$. Then the set of homotopy classes of lifts $g_0$ to $P_1$, as depicted in \Cref{diagram:lifting-aux},
is a torsor for the quotient of $\pi_1(C^X,0)$ by the subgroup generated by
elements of the form \[(\pi_1w_*  m_* ) (x,y),\]
as $x$ ranges over $\pi_1((\Omega K)^X,0)$.
\end{lemma}
\begin{proof} Note that we get a fiber sequence of pointed spaces $(P_0^X,g_0) \to (B^X,g) \to (K^X,0)$. On homotopy, this gives an exact sequence
\[ \pi_2(K^X,0) \to \pi_1(P_0^X,g_0) \to  \pi_1(B^X,g).\]
We have an identification $\pi_2(K^X,0) \cong \pi_1((\Omega K)^X,0) $ under which the map $\pi_2(K^X,0) \to \pi_1(P_0^X,g_0)$ is given by $\pi_1m_*(-,g_0)$, where $m$ is as \Cref{def:tildem1}, $m_*$ is the map on function spaces given by postcomposition with $m$, $\pi_1m_*$ is the induced map on first homotopy, and $g_0\in \pi_1(P_0^X,g_0)$ denotes the constant loop at $g_0\: X \to P_0$. 

Since the image of $\pi_1(P_0^X,g_0) \to \pi_1(B^X,g)$ is zero, the map $\pi_1m_*(-,g_0)$ is surjective and the action
\[\pi_1m_*\: \pi_1((\Omega K)^X,0) \times \pi_1(P_0^X,g_0) \to \pi_1(P_0^X,g_0)\]
induced by $m$ is transitive. So, to compute the image of $\pi_1w_*$, it suffices to fix one element $y \in \pi_1(P_0^X,g_0)$ and compute $(\pi_1w_*m_* )(x,y)$ as $x$ varies over $\pi_1((\Omega K)^X,0)$. 
\end{proof}
\begin{rmk} Note, in particular, that \Cref{lem:lift-22} applies in the case that $\pi_1(B^X,g)=0$.\end{rmk}

When $C$ is an Eilenberg--Mac Lane space, we can interpret the conclusion of \Cref{lem:lift-22} cohomologically.

\begin{cor}\label{lem:lift-2} With set-up as in \Cref{lem:lift-22}, suppose that $C=\prod_iK(G_i,l_i)$ is a finite product of Eilenberg--Mac Lane spaces, so that we may view the homotopy class
\[w\:P_0\to C\] as an element in the product of cohomology groups $ \prod_iH^{l_i}(P_0,G_i).$
Then, under the identification \[\pi_1(C^X,0)\cong \prod_iH^{l_i-1}(X,G_i)\cong \prod_iH^{l_i}(\Sigma X,G_i),\] the image of $\pi_1w_*$  can be computed as follows. 
\begin{itemize} 
\item Fix any $y\: S^1 \times X \to P_0$ representing an element in $\pi_1(P_0^X,g_0)$;
\item Let $x$ vary over $[\Sigma X, \Omega K],$ where we represent homotopy classes as maps $S^1 \times X \to \Omega K$ that restrict to the constant map at the base point on $\{*\} \times X$ and $S^1 \times \{*\}$.
\item The image of $\pi_1w_*$ consists of all homotopy classes of the form 
\[ S^1 \times X \xrightarrow{(x,y)} \Omega K \times P_0 \xrightarrow{m} P_0 \xrightarrow{w} C,\] which in fact give maps $\Sigma X \to C$.
\end{itemize}
In other words, the image of $\pi_1w_*$ can be computed cohomologically as the image of all classes of the form 
\[(x,y)^*m^*w\in \prod_{i}H^{l_i}(\Sigma X,G_i),\] for an appropriate choice of $y\: S^1 \times X \to P_0$ fixed and $x$ varying over $[\Sigma X,\Omega K].$
\end{cor}
\begin{proof} To compute $\pi_1w_*m_*$ we represent the classes $x$ and $y$ from 
\Cref{lem:lift-22} as $x\:S^1 \times X \to \Omega K$ and $y\: S^1 \times X \to P_0$. The requirement that these represent elements in $\pi_1((\Omega K)^X,0)$ and $\pi_1(P_0^X,g_0)$, respectively, precisely amounts to the condition that $x$ restricts to a constant map on $\{*\} \times X$ and $S^1 \times \{*\}$, and that $y$ restricts to $g_0$ on $\{*\} \times X$. With this in hand, elements in the image of $\pi_1w_*$ in $\pi_1(C^X,0)$ are precisely the composites indicated in the displayed equation of the third item. 

To see that such composites induce a map from the suspension, note that $w\circ m\circ (x,y)$ is homotopic to $w \circ g_0\cong 0$ when restricted to $\{*\}\times X$, and is constant when restricted to $S^1\times\{*\}$, since both $x$ and $y$ restrict to the basepoint there. The description in terms of cohomology is immediate.
\end{proof}

\begin{rmk}\label{rmk:practice} With the set-up of \Cref{lem:lift-2}, it is in practice usually easiest to determine the image of $\pi_1w_*$ by computing \[(x,y)^*m^*w-(0,y)^*m^*w\] as $x$ varies over $[\Sigma X,\Omega K].$ \end{rmk} 

We also note another lifting-type lemma, which is an elementary consequence of classical Postnikov theory, but is useful to codify.

\begin{lemma}\label{lem:lift-pin} Suppose that $f\:X \to Y$ is a map of simply connected finite-type spaces so that  the $k$-truncation $\tau_{\leq k}f\: \tau_{\leq k} X\to \tau_{\leq k}Y$ is a homotopy equivalence for some $k\geq 2$, and that $Y$ sits in a fiber sequence 
\[ Y \to \tau_{\leq k} Y \xrightarrow{k} K(\pi_{k+1}Y,k+2).\] 
Then $k\: \tau_{\leq k}Y \to K(\pi_{k+1}Y,k+2)$ lifts to $\tilde{k}$ as indicated in \Cref{diagram-tildek} below. 
\begin{equation}\label[diagram]{diagram-tildek}
\begin{tikzcd}
& K(\pi_{k+1}X,k+2)\ar[d,"{K(\pi_{k+1}f,k+2)}"]\\
\tau_{\leq k}Y \ar[r,"k"] \ar[ur, dashed,"\tilde{k}"]& K(\pi_{k+1}Y,k+2).\end{tikzcd}
\end{equation}
Moreover, $\tau_{\leq k+1} X = \fib\left( \tilde{k}\: \tau_{\leq k} Y \to K(\pi_{k+1}X,k+2)\right)$.
\end{lemma}

\section{Preliminary results}\label{structure}
We next study relationships between different bundle enumeration problems.  These results substantially reduce the complexity and length of subsequent calculations. 

We begin by explaining a useful action of homotopy groups on homotopy classes of maps.
\begin{const}\label{ACTION} Let $X$ be a simply connected $k$-dimensional CW complex. Fix a top-dimensional cell $D$. Let $D'\subset D$ be a closed ball in the interior of $D$, with boundary $\partial D'$. Let $p_D\: X \to X \vee S^{k}$ denote the natural map $X \to X/\partial D'$. Given any simply connected space $Y$, we obtain an action 
\[\pi_{k}Y \times [X,Y]\to [X,Y]\] given by $\left( [\sigma\: S^{k} \to Y],[f\: X \to Y] \right)\mapsto [\sigma f\: X \xrightarrow{p_D} X \vee S^{k} \xrightarrow{f \vee \sigma} Y].$
\end{const}
\begin{rmk} We assume all spaces are simply connected to avoid basepoint issues. For non-simply-connected spaces, the analogous construction applies to homotopy classes of pointed maps. \end{rmk}

Applying \Cref{ACTION} with $Y=BU(n)$ gives a useful way to modify vector bundles, often in a way that does not change the $K$-theory class of the bundle. Suppose that $X$ has top cell in dimension $k$, $k>2n$, and that $V\: X \to BU(n)$ is given. Then $\sigma V$ and $V$ have the same Chern classes. If $X$ has even cells, they also have the same $K$-theory class. Thus, if $X$ has even cells through dimension $k>2n$, \Cref{ACTION} induces an action of $\pi_{k}BU(n)$ on $\Vect^h_n(X)$. We explore this further in the case of $X=\CP^{n+2}$.

\begin{lemma}\label{lem1} Let $\eta$ be a vector bundle over $\CP^{n+2}$ with $K$-theory class $h \in [\CP^{n+2},BU]$.  Then $\eta'\in \Vect^h_n(\CP^{n+2})$ if and only if $\eta' \simeq \sigma \eta$ for some $\sigma \in \pi_{2n+4}BU(n)$, where the action is as described in \Cref{ACTION}.
\end{lemma}
\begin{proof} Note that the action of $\pi_{2n+4}BU(n)$ on a given map $\eta\: \CP^{n+2} \to BU(n)$ can be identified with the transitive action of $\pi_{2n+4}BU(n)$ on extensions of \[ \eta|_{\CP^{n+1}}\: \CP^{n+1} \to BU(n)\] over the cellular inclusion 
$\CP^{n+1} \to \CP^{n+2}$. 
By \cite[Theorem 4.3]{opie24enum}, any two elements in $\Vect^{h}_n(\CP^{n+2})$ have the same restriction to $\CP^{n+1}$, and differ by an element in $\pi_{2n+4}BU(n)$. \end{proof}

Next, we explore how \Cref{lem1} helps us to analyze stably trivial rank $n$ vector bundles on $\CP^{n+2}$ and relate them to non-stably-trivial bundles. 

\begin{cor}\label{Cor1} Every stably trivial rank $n$ vector bundle on $\CP^{n+2}$ factors through the map 
\[q\:\CP^{n+2} \to  \CP^{n+2}/\CP^{n+1} \simeq S^{2n+4}.\] \end{cor}
\begin{proof} By \Cref{lem1}, stably trivial rank $n$ bundles on $\CP^{n+2}$ restrict to zero on $\CP^{n+1}$.
\end{proof}

Now consider $V\: \CP^{n+1} \to BU(n)$, a rank $n$ representative for $h|_{\CP^{n+1}}$, which exists by \cite[Theorem 1.3]{opie24enum}. We study whether or not $V$ extends over $\CP^{n+2}$, which would give a rank $n$ representative for $h$. 

\begin{lemma}\label{lem:exists-n-odd-c1-even} Let $h\: \CP^{n+2} \to BU$ be given. If $n$ is odd, or if $n$ is even and $c_1(h)\equiv 0 \pmod{2}$, then $h$ admits a rank $n$ representative if and only if $c_{n+1}(h)=c_{n+2}(h)=0$.
\end{lemma}
\begin{proof}

The obstruction to extending a rank $n$ vector bundle $V$ on $\CP^{n+1}$ over $\CP^{n+2}$ is the composite
\[ S^{2n+3} \xrightarrow{\eta} \CP^{n+1} \xrightarrow{V} BU(n),\]
where $\eta$ is the Hopf map. Consider the diagram:
\[\begin{tikzcd}[row sep=1.2em]
& &S^{2n+1} \ar[d,"x"] \\
S^{2n+3} \ar[urr,dashed,crossing over,bend left=15,"b"]\ar[r,"\eta"]& \CP^{n+1} \ar[r,"V"] &BU(n)\ar[d,"y"]\\
  & & BU(n+1) \end{tikzcd}
\]
where the right-most vertical triple is a fiber sequence. The dashed arrow labeled $b$ exists, satisfying $x \circ b \cong V \circ \eta$, since every rank $n+1$ bundle on $\CP^{n+1}$ extends over $\CP^{n+2}$ \cite[Theorem 1.3(ii)]{opie24enum}.

In the case that $n$ is odd, the map $S^{2n+1} \to BU(n)$ induces the zero map on $\pi_{2n+3}$ since the natural morphism $\pi_{2n+3}BU(n) \to \pi_{2n+3}BU(n+1)$ is injective (use \cite[pp. 970 - 971]{Mimura_HBAT} to consider the long exact sequence on homotopy associated to the fibration).

In the case that $n$ is even and $c_1(h)$ is even, there are two non-isomorphic representatives for $h|_{\CP^{n+1}}$ by \cite[Theorem 1.3(ii)]{opie24enum}. By  \cite[Theorem 1.3(iii)]{opie24enum}, exactly one of these extends over $\CP^{n+2}$.
\end{proof}
\begin{rmk} By \cite[p. 152]{AR}, if $n=2$, \Cref{lem:exists-n-odd-c1-even} is also true for $c_1(h)$ odd: every rank $2$ bundle on $\CP^{3}$ with odd first Chern class extends over $\CP^4$. The proof makes use of some explicit computations in twisted symplectic $K$-theory, and it is not clear how to adapt this argument to even $n=4$.\end{rmk}

\section{Computations at the prime $3$}\label{MoorePost3}

In this section, we focus on the prime $p=3$. Some cases follow easily from the theory developed in the previous section.
\begin{prop} \label{prop:prime3easycase} Let $h\: \CP^{n+2} \to BU \to \c{BU}{3}$ be given, and suppose that $h$ admits a $3$-complete rank $n$ representative.
When $n\equiv 1, 2 \mod 3$, 
$\c{\Vect^h_n(\CP^{n+2})}{3}$ is a singleton. 
\end{prop}
\begin{proof} If $n\equiv 1,2 \pmod{3}$, then $\c{(\pi_{2n+4}BU(n))}{3}\cong 0$ by \cite[page 970]{Mimura_HBAT}. By the $3$-complete version of \Cref{lem1}, we find that there is a transitive action of $\c{(\pi_{2n+4}BU(n))}{3}\cong 0$ on $\c{\Vect^h_n(\CP^{n+2})}{3}.$ So, $\c{\Vect^h_n(\CP^{n+2})}{3}$ is either empty or a singleton. Under the hypothesis that $h$ admits a $3$-complete rank $n$ representative, the set must be a singleton.
\end{proof} 
\begin{rmk} In fact, it is straightforward to show directly from an obstruction-theoretic argument that $h$ admits a $3$-complete rank $n$ representative if and only if the top two $3$-completed Chern classes vanish. However, in the main case of interest, $n$ is odd and we already have a proof of this integrally (cf. \Cref{lem:exists-n-odd-c1-even}).
\end{rmk} 
We now focus on the case $n\equiv 0 \pmod{3}$. Let $F$ be the homotopy fiber of the map
\[
c_{n+1}\times c_{n+2}\: BU \to K(\Z,2n+2) \times K(\Z, 2n+4)
\] 
where $c_{n+1}$ and $c_{n+2}$ denote, respectively, the operations of taking $(n+1)$-st and $(n+2)$-st Chern classes.

The space $F$ accounts for the primary Chern class obstructions to finding a rank $n$ representative for $h$. When $n\equiv0\pmod3$, the cohomology of $F$ contains one further obstruction class $U$. Passing to $G=\operatorname{fib}(U)$ gives the required approximation to $\c{BU(n)}{3}$, and the Moore--Postnikov lifting calculus described in \Cref{lifting} converts the enumeration into a computation of $\operatorname{coker}(\pi_1U_*).$

We choose a lift $f\: BU(n) \to F$ of the natural map $BU(n) \to BU$:
\begin{equation}\label[diagram]{diag:f}
\begin{tikzcd}
& F\ar[d] \\
BU(n) \ar[r]\ar[ur, dashed,"\exists f"]& BU \ar[rr,"c_{n+1} \times c_{n+2}"] & & K(\Z,2n+2) \times K(\Z, 2n+4).
\end{tikzcd} 
\end{equation}
We compute the Serre spectral sequence associated with the fiber sequence
\begin{equation} \label[diagram]{fibF}
K(\Z,2n+1)\times K(\Z, 2n+3) \to F \to BU,
\end{equation}
which takes the form
$$E_2^{t,q}=\HZ{3}^t\left(BU, \HZ{3}^q(K(\Z,2n+1)\times K(\Z, 2n+3))\right)\implies \HZ{3}^{t+q}F.$$
The $E_2$-page is depicted in \Cref{fig2}. The differentials on the $E_r$-page are $d_r^{t,q}\:E_r^{t,q} \to E_r^{t+r,q-r+1}.$ 
\begin{prop}[\cite{MT}]
Let $\iota_j$ be the standard generator for $\HZ{3}^{j}K(\Z,j)$. Then
\[\HZ{3}^{*\leq 2n+5}K(\Z,2n+1)\cong \Z/3\{1,\iota_{2n+1}, P^1\iota_{2n+1}\}. \]
\end{prop}

\begin{figure}[h]
\centering
\textbf{The $E_2$-page for the Serre spectral sequence computing $\HZ{3}^*F$}\par\medskip
\begin{tikzpicture}
\matrix (m) [matrix of math nodes,
nodes in empty cells,nodes={minimum width=2ex,
minimum height=2ex,outer sep=-1pt},
column sep=.2ex,row sep=.2ex]{
t&\\
2n+5 && P^1\iota_{2n+1} & & & & & & & & & & && \\
2n+4 && & & & & & & & & & & && \\
2n+3 && \iota_{2n+3} & & & & & & & & & & && \\
2n+2 && & & & & & & & & & & && \\
2n+1 && \iota_{2n+1} & &  & & & & & & & & && \\
\vdots\\
0&&1 & & & & \hdots &&  c_{n+1} &  & c_{n+2} & & c_{n+3}\\
&& & & & & & & & & & & && \\
\quad\strut & & 0  & &  && \hdots& & 2n+2 & &2n+4& & 2n+6 & & q \strut \\};
\draw[->] (m-4-3.south east) -- (m-8-11.north west);
\draw[->] (m-6-3.south east) -- (m-8-9.north west);
\draw[thick] (m-1-2) -- (m-10-2) ;
\draw[thick] (m-9-1) -- (m-9-14) ;
\end{tikzpicture}
\caption{A schematic of the $E_2$-page for the $\Z/3$-cohomology Serre spectral sequence for the fiber sequence $K(\Z,2n+1)\times K(\Z, 2n+3) \to F \to BU$. Only multiplicative generators are shown. Arrows indicate differentials on key classes.}\label{fig2}
\end{figure}
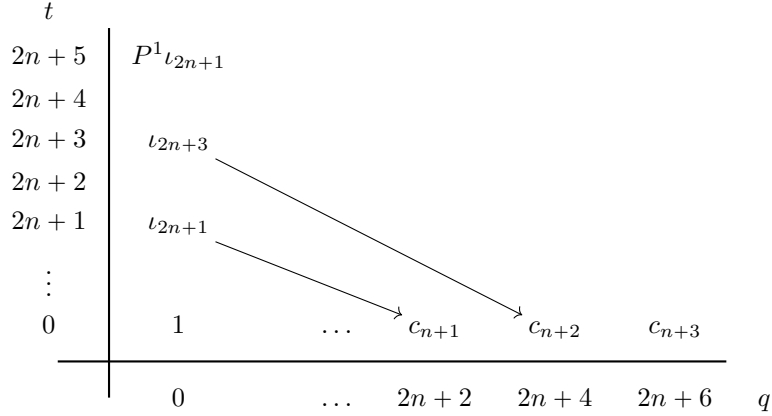

Note that since $F$ is defined via the fiber sequence \Cref{fibF}, \[
d_{2n+2}(\iota_{2n+1}) = c_{n+1}, \text{ and } d_{2n+4}(\iota_{2n+3}) = c_{n+2}.\]
By transgression and \Cref{prop:Chernform3},
\[
d_{2n+6}(P^1\iota_{2n+1}) = P^1d_{2n+2}(\iota_{2n+1}) = P^1c_{n+1}=
c_1^2c_{n+1} + c_2c_{n+1} +2c_1c_{n+2} + (n+3)c_{n+3}.
\]
Since $c_{n+1}$ and $c_{n+2}$ are both killed in previous pages, and $n\equiv 0 \pmod{3}$, $
d_{2n+6}(P^1\iota_{2n+1}) =
0 .$ 

Let $U \in \HZ{3}^{2n+5}F$ be the class detected by $P^1\iota_{2n+1}$. 
Note that $U$ is represented by
$P^1\iota_{2n+1}\otimes 1 - \iota_{2n+1}\otimes c_1^2 - \iota_{2n+1}\otimes c_2 + \iota_{2n+3}\otimes c_1$
on the $E_2$-page of the Serre spectral sequence for 
\[
K(\Z, 2n+1)\times K(\Z, 2n+3) \to F \to BU,
\]
where $\iota_j$ generates $\HZ{3}^jK(\Z,j)$.

Let $G$ denote the homotopy fiber of $U$. Given $BU(n) \to F$ inducing an isomorphism on homotopy through degree $2n+4$, let $g\: BU(n) \to G$ be a lift fitting into \Cref{diag:g}.

\begin{equation}\label[diagram]{diag:g}
\begin{tikzcd}
& G\ar[d] & &  \\
& F\ar[d] \ar[r,"U"] & K(\Z/3,2n+5) & & \\
BU(n) \ar[r] \ar[ur, bend left = 5, "f", near end] \ar[uur, bend left = 20, dashed," g"] & BU \ar[rr,"c_{n+1}\times c_{n+2}"] & & K(\Z,2n+2)\times K(\Z, 2n+4).
\end{tikzcd} 
\end{equation}

By \Cref{lem:basic}, we deduce:

\begin{prop}\label{prop:prime3hardcase} 
Let $n\equiv 0 \mod 3$. There is an equivalence 
\[
\tau_{\leq 2n+4}g\: \tau_{\leq 2n+4} \c{BU(n)}{3} \to \tau_{\leq 2n+4}\c{G}{3},
\]
and hence a bijection
\[
[\CP^{n+2}, \c{BU(n)}{3}] \cong [\CP^{n+2}, \c{G}{3}].
\]
\end{prop}
\begin{thm}\label{main3count}
Let $n$ be odd and let $h\: \CP^{n+2} \to BU$ be such that $c_{n+1}(h)=c_{n+2}(h)=0$. If $n\equiv 0 \mod 3$ and $c_{1}(h) \equiv c_{2}(h) \equiv 0 \pmod 3$, then the composite
\[\c{h}{3}\:\CP^{n+2} \xrightarrow{h}BU \to \c{BU}{3}\] admits three non-homotopic lifts to $\c{BU(n)}{3}$. Otherwise, $\c{h}{3}$ admits a unique lift to $\c{BU(n)}{3}.$
\end{thm}
\begin{proof} By \Cref{lem:exists-n-odd-c1-even}, $h$ admits a rank $n$
representative, and hence a $3$-complete rank $n$ representative.
Thus, \Cref{prop:prime3easycase} handles the cases $n\equiv 1,2\pmod{3}$.

For the remaining case $n\equiv 0\pmod{3}$, we use \Cref{prop:prime3hardcase}. Given a lift $f\: \CP^{n+2} \to F$ of $h$, the obstruction to lifting to 
$G$ lies in  $\HZ{3}^{2n+5}\CP^{n+2}=0.$
Fix a choice of $f$, and consider
\begin{equation}\label{eq:pi1u} 
\pi_1U_* \: \pi_1(F^{\CP^{n+2}},f) \to \pi_1(K(\Z/3,2n+5)^{{\CP^{n+2}}},0).
\end{equation}  

By \Cref{cor:coker-1}, homotopy classes of lifts of $f\: \CP^{n+2} \to F$ to $G$ are a torsor for $\op{coker}(\pi_1U_* )$.  
To enumerate lifts of $f$ to $G$, we use \Cref{lem:lift-2} with fiber sequences
\[ G \to F \xrightarrow{U} K(\Z/3,2n+5),\,\,\text{   and    }\,\,F \to BU \to  K(\Z,2n+2) \times K(\Z,2n+4) ,\]
as \Cref{eq:fib1},
and \Cref{eq:fib2}, respectively. 
If we fix $A\: S^1 \times  \CP^{n+2} \to F$ representing an element in $\pi_1(F^{\CP^{n+2}},f)$, the image of $\pi_1U_*$ is precisely the set of homotopy classes $U \circ \lambda$ represented in \Cref{cd:action} below:
\begin{equation}\label[diagram]{cd:action}
\begin{tikzcd}
{S^1 \times \CP^{n+2}}\ar[r, dashed, "\lambda"] \ar[dr, bend right = 5, "{(x, y, A)}\,\,\,\,\,\,\,\,\,\,\,\," below]
& F \ar[r,"U"] 
& K(\Z/3,2n+5) \\
& {K(\Z,2n+1) \times K(\Z,2n+3) \times F} \ar[u,"m"],
\end{tikzcd}
\end{equation}
as $x,y$ vary over $[\Sigma \CP^{n+2}, K(\Z,2n+1) \times K(\Z,2n+3)].$

Let $K$ denote $K(\Z,2n+1) \times K(\Z,2n+3)$. We have a diagram of fibrations
\begin{equation}\label[diagram]{map_fibs}\begin{tikzcd}
K \times K \ar[r]\ar[d,"m_1"]& K \times F \ar[d,"m"]\ar[r]&BU\ar[d,"="] \\
K \ar[r] & F \ar[r] & BU,
\end{tikzcd}\end{equation}
which induces a map on Serre spectral sequences.
On the $E_2$-page:
\begin{align*}
 & m^*(P^1\iota_{2n+1}\otimes 1 - \iota_{2n+1}\otimes c_1^2 - \iota_{2n+1}\otimes c_2 + \iota_{2n+3}\otimes c_1) \\
 & = P^1\iota_{2n+1}\otimes 1 \otimes 1 
+ 1\otimes P^1\iota_{2n+1}\otimes 1 
- 1\otimes  \iota_{2n+1}\otimes c_1^2 
-  \iota_{2n+1}\otimes 1 \otimes c_1^2 \\
 & \quad\quad -1 \otimes \iota_{2n+1}\otimes c_2 -\iota_{2n+1}\otimes 1 \otimes c_2 +1\otimes  \iota_{2n+3}\otimes c_1 + \iota_{2n+3}\otimes 1 \otimes c_1.
\end{align*}
Thus,
\[
m^*U=U\otimes 1 + P^1\iota_{2n+1}\otimes 1 - \iota_{2n+1}\otimes c_1^2 - \iota_{2n+1}\otimes c_2 +  \iota_{2n+3}\otimes c_1
.\]
Write $\iota_1t^i$ for the degree $2i+1$ generator of $\HZ{3}^*\Sigma \CP^{n+2}$. 
Let $x, y$ be integer multiples of $\iota_1t^n$ and $\iota_1t^{n+1}$, respectively, so that $x=x_0\iota_1t^n$ and $y=y_0\iota_1t^{n+1}$.
Then: 
\begin{align*}
 (x,y,A)^*(m^*U ) - (0, 0, A)^*(m^*U) &= P^1(x_0\iota_1t^n) - (x_0\iota_1t^n)c_1(h)^2t^2 - (x_0\iota_1t^n)c_2(h)t^2 \\
& \quad \quad + (y_0\iota_1t^{n+1})c_1(h)t \\
 &= nx_0\iota_1 t^{n+2} - x_0\iota_1c_1(h)^2t^{n+2} - x_0\iota_1c_2(h)t^{n+2} + y_0\iota_1c_1(h)t^{n+2} \\
 &=(nx_0-x_0c_1(h)^2-x_0c_2(h)+y_0c_1(h))\iota_1t^{n+2} \\
 &=(y_0c_1(h)-x_0c_1(h)^2-x_0c_2(h))\iota_1t^{n+2}.
\end{align*}
If $c_1(h), c_2(h)$ are both zero modulo 3, the above is zero regardless of integers $x_0, y_0$. 
 Otherwise, the image as $x_0, y_0$ vary is all of $\Z/3$. 
\end{proof}

\section{Computations at the prime $2$}\label{MoorePost2}


In this section, we focus on computing maps from $\CP^{n+2}$ to $\c{BU(n)}{2}$ lifting a given map \[h\:\CP^{n+2} \to BU \to \c{BU}{2}.\] By \Cref{cor:product23}, this amounts to calculating the power of $2$ dividing the size of $\Vect^h_{n}(\CP^{n+2})$. Our computations break up based on $n\pmod{4}$ and the Chern classes of $h$ modulo $4$, with some dependence on $n\pmod{8}$.

In \Cref{subsec:minimal}, we deal with cases that require minimal work by appealing to our preliminary results from \Cref{structure}.  
We consider the remaining harder case $n\equiv 1\pmod{4}$ in \Cref{sec:n1mod4}. The latter section uses obstruction-theoretic approaches to lifting problems. Depending on $n\pmod{8}$, we study two variations on a Moore--Postnikov-type factorization of the morphism $\c{BU(n)}{2} \to \c{BU}{2}$. Choosing the right factorization is important for solving the problem. We give a more detailed roadmap of the set-up for each case in \Cref{rmk:roadmap} below.

\begin{convention}\label{complete-conv}
Throughout the remainder of the paper, all completions are at the prime $2$. To simplify notation, for any space or group $X$, we write $\c{X}{}$ in place of $\c{X}{2}$.
\end{convention}

\begin{rmk}\label{rmk:roadmap}
The $2$-primary obstruction-theoretic argument can be outlined as follows:
\begin{itemize}
\item If $n\equiv 3,7\pmod{8}$, the preliminary results from \Cref{structure} determine the answer without any Moore--Postnikov calculations.

\item If $n\equiv 1\pmod{8}$, we use a tower of principal fibrations:
\begin{equation}\label[diagram]{towerx}
Y\longrightarrow X\longrightarrow F\longrightarrow BU,
\end{equation}
where the fibration defining $F$ has as its base a product of integral Eilenberg--Mac Lane spaces, and those defining the higher stages have $\Z/2$-Eilenberg--Mac Lane spaces as their bases. Lifts of $h$ to $\c{Y}{}$ are in bijection with lifts to $\c{BU(n)}{}$.

\item If $n\equiv 5\pmod{8}$, we modify the tower from the case $n\equiv 1\pmod{8}$ in two ways. First, since
\[
\c{\pi_{2n+4}BU(n)}{}\cong
\c{\Vect_n^0(\CP^{n+2})}{}\cong \Z/4,
\]
we must replace the last $\Z/2$-valued obstruction arising in \Cref{towerx} by a $\Z/4$-valued lift. To make this work computationally, we also need to arrange that $c_1(h)$ is even. This is possible since $n$ is odd: by tensoring with a suitable line bundle, we reduce to the case $c_1(h)\equiv 0\pmod{4}$, and hence in particular to the case in which $c_1(h)$ is even. It is therefore convenient to modify the entire tower \Cref{towerx} and construct a tower of principal fibrations:
\begin{equation}\label[diagram]{towery}
\widetilde{Y}_2\longrightarrow X_2\longrightarrow F_2\longrightarrow BU_2,
\end{equation}
where $BU_2$ classifies stable vector bundles with even first Chern class. In \Cref{towery}, the base of the fibration defining $F_2$ is a product of integral Eilenberg--Mac Lane spaces; the base of the fibration defining $X_2$ is a $\Z/2$-Eilenberg--Mac Lane space; and the base of the fibration defining $\widetilde{Y}_2$ is a $\Z/4$-Eilenberg--Mac Lane space.

Lifting to $\c{(\widetilde{Y}_2)}{}$ is sufficient to classify $2$-complete rank $n$ vector bundles on $\CP^{n+2}$ with even first Chern class. However, lifting along the last stage in \Cref{towery} necessitates some fairly involved $\Z/4$-cohomology computations.
\end{itemize}
\end{rmk}

\subsection{Easy cases}\label{subsec:minimal} We begin by treating as many cases as possible with tools from \Cref{structure}. If $n\equiv 3 \pmod{4}$, then $\c{(\pi_{2n+4}BU(n))}{}\cong 0$ by \cite[page 970]{Mimura_HBAT}. By the $2$-complete version of \Cref{lem1}, we find that there is a transitive action of $\c{(\pi_{2n+4}BU(n))}{}\cong 0$ on $\c{\Vect^h_n(\CP^{n+2})}{},$ which is moreover non-empty by \Cref{lem:exists-n-odd-c1-even}. Therefore $\c{\Vect^h_n(\CP^{n+2})}{} $ is a singleton. We deduce that:
\begin{thm}\label{thm:n3mod4}Let $n\equiv 3\pmod{4}$ and let $h\: \CP^{n+2}\to BU$ be a reduced complex $K$-theory class with $c_{n+1}(h)=c_{n+2}(h)=0$. Then 
\[
\#\Vect^h_n(\CP^{n+2})=\begin{cases} 3 & \text{ if } n\equiv c_1(h) \equiv c_2(h) \equiv 0 \pmod{3};\\
1 & \text{ otherwise.} \end{cases}\]
\end{thm}

\subsection{The case $n\equiv 1 \pmod{4}$}\label{sec:n1mod4}


The relevant Steenrod square actions on Chern classes are
\[
\Sq^2c_{n+1} = c_1c_{n+1}+c_{n+2}, \quad \Sq^2c_{n+2} = c_1c_{n+2}, \quad \text{and} \quad \Sq^4c_{n+1} = c_2c_{n+1},
\]
by \Cref{prop:Chernform2}. We form the Serre spectral sequence for the fiber sequence
\[ K(\Z,2n+1) \times K(\Z,2n+3)\to F \to BU,\]
where $F$ is the fiber of $c_{n+1} \times c_{n+2}$. The $E_2$-page is depicted in \Cref{fig7}.

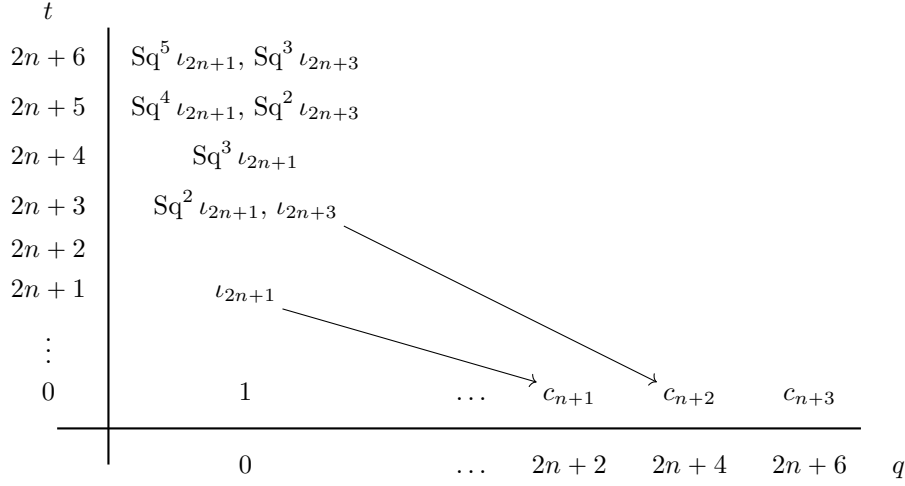
\begin{figure}[h]
\centering
\textbf{The $E_2$-page for the Serre spectral sequence computing $\HZ{2}^*F$}\par\medskip
\begin{tikzpicture}
\matrix (m) [matrix of math nodes,
nodes in empty cells,nodes={minimum width=2ex,
minimum height=2ex,outer sep=-1pt},
column sep=.2ex,row sep=.2ex]{
t&\\
2n+6 && \Sq^5\iota_{2n+1},\,\Sq^3\iota_{2n+3}& & & & & & & & & & && \\
2n+5 && \Sq^4\iota_{2n+1},\,\Sq^2\iota_{2n+3} & & & & & & & & & & && \\
2n+4 &&\Sq^3\iota_{2n+1} & & & & & & & & & & && \\
2n+3 &&\Sq^2\iota_{2n+1},\, \iota_{2n+3} & & & & & & & & & & && \\
2n+2 && & & & & & & & & & & && \\
2n+1 && \iota_{2n+1} & &  & & & & & & & & && \\
\vdots\\
0&&1 & & & & \hdots &&  c_{n+1} &  & c_{n+2} & & c_{n+3}\\
&& & & & & & & & & & & && \\
\quad\strut & & 0  & &  && \hdots& & 2n+2 & &2n+4& & 2n+6 & & q \strut \\};
\draw[->] (m-5-3.south east) -- (m-9-11.north west);
\draw[->] (m-7-3.south east) -- (m-9-9.north west);
\draw[thick] (m-1-2) -- (m-11-2) ;
\draw[thick] (m-10-1) -- (m-10-14) ;
\end{tikzpicture}
\caption{A schematic of the $E_2$-page for the $\Z/2$-cohomology Serre spectral sequence for the fiber sequence $K(\Z,2n+1)\times K(\Z, 2n+3) \to F \to BU$. Only multiplicative generators are shown. Arrows indicate differentials on key classes.}\label{fig7}
\end{figure}
 We find that
\begin{itemize}
 \item $d_{2n+4}(\Sq^2\iota_{2n+1}) = \Sq^2d_{2n+2}(\iota_{2n+1}) = \Sq^2c_{n+1} = c_1c_{n+1}+c_{n+2}$,
 \item $d_{2n+5}(\Sq^3\iota_{2n+1}) = \Sq^3d_{2n+2}(\iota_{2n+1}) = \Sq^3c_{n+1} = 0$,
 \item $d_{2n+6}(\Sq^2\iota_{2n+3}) = \Sq^2d_{2n+4}(\iota_{2n+3}) = \Sq^2c_{n+2} = c_1c_{n+2}$, and
 \item $d_{2n+6}(\Sq^4\iota_{2n+1}) = \Sq^4d_{2n+2}(\iota_{2n+1}) = \Sq^4c_{n+1} = c_2c_{n+1}$.
\end{itemize}
It follows that 
\[
\Sq^2\iota_{2n+1}+\iota_{2n+3}, \quad \Sq^3\iota_{2n+1}, \quad \Sq^2\iota_{2n+3}, \quad \Sq^4\iota_{2n+1},
\]
are permanent cycles that detect classes $u,v,m,w$, respectively, in $\HZ{2}^{*}F$, which are represented in the double complex as follows:
\begin{align*}
u &=[\Sq^2\iota_{2n+1}\otimes 1 + \iota_{2n+1} \otimes c_1 + \iota_{2n+3}\otimes 1] & v= [\Sq^3\iota_{2n+1}\otimes 1] \\
m&= [\Sq^2\iota_{2n+3}\otimes 1 + \iota_{2n+3} \otimes c_1 ] & w=[\Sq^4\iota_{2n+1} \otimes 1  +\iota_{2n+1} \otimes c_2].
\end{align*} 
By \cite[1.5]{singer73}, we find that, up to terms of lower filtration,
\[
 \Sq^1u=v,  \quad  \Sq^2u=m, \quad \Sq^1w=\Sq^2v, \quad \text{and} \quad \Sq^3u= \Sq^1m
\]
in the cohomology of $F$.

We now build the next stage of the tower. Consider the map
\[
u \: F \to K(\Z/2, 2n+3).
\]
Write $X$ for its homotopy fiber, which fits into \Cref{diag:g1mod4}.

\begin{equation}\label[diagram]{diag:g1mod4}
\begin{tikzcd}[column sep = large]
& X\ar[d] & &  \\
& F\ar[d] \ar[r,"u "] & K(\Z/2,2n+3)   \\
BU(n) \ar[r] \ar[ur, bend left = 5,dashed] \ar[uur, bend left = 20, dashed] & BU \ar[r,"c_{n+1}\times c_{n+2}"] & K(\Z,2n+2)\times K(\Z, 2n+4).
\end{tikzcd} 
\end{equation}
Consider the Serre spectral sequence for $X$ as shown in \Cref{fig4}.
\begin{figure}[h]
\centering
\textbf{The $E_2$-page of the Serre spectral sequence computing $\HZ{2}^*X$}\par\medskip
\begin{tikzpicture}
\matrix (m) [matrix of math nodes,
nodes in empty cells,nodes={minimum width=2.5ex,
minimum height=2.5ex,outer sep=-1pt},
column sep=.25ex,row sep=.25ex]{
t&\\
2n+4 && \Sq^2\iota_{2n+2} &  & & & & & & & & & &&& \\
2n+3 && \Sq^1\iota_{2n+2} & & & & & & & & & & &&& \\
2n+2 && \iota_{2n+2} & &  & & & & & & & & &&& \\
\vdots\\
0&&1 & & & & \hdots &&  u &  & v & &m,w & \\
&& & & & & & & & & & & && \\
\quad\strut & & 0  & &  && \hdots& & 2n+3 & &2n+4& & 2n+5 & & & q \strut \\};
\draw[->] (m-4-3.south east) -- (m-6-9.north west);
\draw[->] (m-3-3.south east) -- (m-6-11.north west);
\draw[->] (m-2-3.south east) -- (m-6-13.north west);
\draw[thick] (m-1-2) -- (m-8-2) ;
\draw[thick] (m-7-1) -- (m-7-15) ;
\end{tikzpicture}
\caption{A schematic of the $E_2$-page for the $\Z/2$-cohomology Serre spectral sequence for the fiber sequence $K(\Z/2,2n+2) \to X \to F$. Only key multiplicative generators are shown. Arrows indicate differentials.}\label{fig4}
\end{figure}
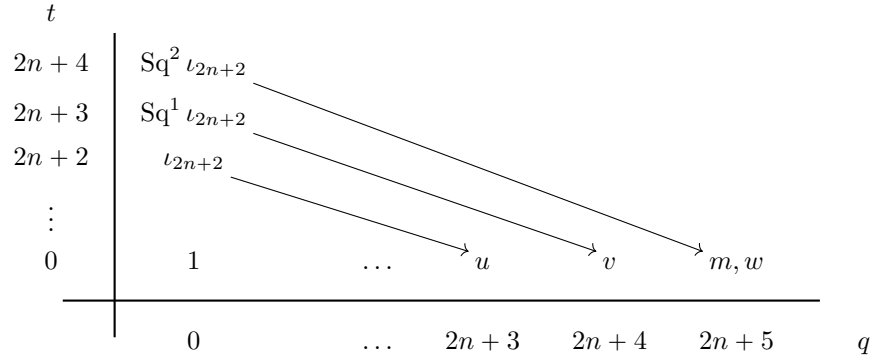
Note that the classes $u,v,m,\Sq^1m$ are all killed by transgressions of Steenrod squares on $\iota_{2n+2}$. We find that 
\[\HZ{2}^{*\leq 2n+6}X\cong\left( \Z/2[c_1,\ldots,c_n]\otimes \Z/2\{1, w\}\right)^{*\leq 2n+6},\]
where $w$ is in degree $2n+5$ and is pulled back from the cohomology of $F$. Define
\begin{equation}\label{def:Y1} Y= \fib \left(w\: X \to K(\Z/2,2n+5)\right).\end{equation}
\begin{lemma}\label{n1mod8-Y-enough} If $n\equiv 1 \pmod{8}$, the natural map $BU(n) \to Y$ induces a bijection \[[\CP^{n+2}, \c{BU(n)}{}]\cong [\CP^{n+2},\c{Y}{}].\]
\end{lemma}
\begin{proof} This follows from \Cref{lem:basic} and the fact that $\pi_{2n+4}BU(n)\cong \Z/2$.
\end{proof}
If $n\equiv 5\pmod{8}$, by \Cref{lem:lift-pin} and the fact that $\pi_{2n+4}BU(n)\cong \Z/4$, $\Sq^1w=0$ and $w$ lifts to a $\Z/4$-cohomology class $\tilde{w}\:X \to K(\Z/4,2n+5).$ We define:
\begin{equation}\label{def:tildeL1} \tilde Y= \fib \left(\tilde w\: X \to K(\Z/4,2n+5)\right).\end{equation}

\begin{lemma}\label{n1mod8-tildeL-enough} If $n\equiv 5 \pmod{8}$, there is a natural map $BU(n) \to \tilde{Y}$ that induces a bijection \[[\CP^{n+2}, \c{BU(n)}{}]\cong [\CP^{n+2},\c{\tilde{Y}}{}].\] Moreover, the composite $BU(n) \to Y$ induces a surjection \[ [\CP^{n+2},\c{BU(n)}{}] \twoheadrightarrow[\CP^{n+2},\c{Y}{}].\]
\end{lemma}
\begin{proof} The first statement follows from \Cref{lem:basic}, observing that $\pi_{2n+4}BU(n)\cong \Z/4$. The second follows from the fact that we have a fiber sequence $\tilde{Y} \to Y \to K(\Z/2,2n+5).$ 
\end{proof}

We would like to compute lifts of a given $K$-theory class to $Y$ or $\tilde{Y}$ via an action of some Eilenberg--Mac Lane space on $X$ (so that we can apply \Cref{lem:lift-0} and \Cref{lem:lift-2}). However, there are some subtleties.
For example, in the case that $n\equiv 5 \pmod{8}$, we have a diagram of homotopy pullbacks
\begin{equation}\label[diagram]{issue:1mod2}
\begin{tikzcd}
\fib(\Sq^2+ \iota_{2n+3})\ar[d]\ar[r]&X\ar[d]\ar[r] & * \ar[d]\\
K(\Z,2n+1) \times K(\Z,2n+3) \ar[rr, bend right=10, "\Sq^2 + \iota_{2n+3}" {near start, below}]\ar[d]\ar[r]& F \ar[r,"u"]\ar[d,crossing over] & K(\Z/2,2n+3) \\
* \ar[r] & BU ,
\end{tikzcd}
\end{equation}
However, we cannot realize $X\to BU$ as the fiber of some map to a delooping of
 $\fib(\Sq^2 + \iota_{2n+3})$.  We work around this problem in the next few sections. 

\subsubsection{The case $n\equiv 1\pmod{8}$}\label{n1mod8-start}

With notation as introduced in the previous section, note that
\[ Y=\fib\left( u \times w\: F \to K(\Z/2,2n+3) \times K(\Z/2,2n+5) \right).\]

Thus, we can analyze the lifting problem of \Cref{diag:g1mod4-22} below.

\begin{equation}\label[diagram]{diag:g1mod4-22}
\begin{tikzcd}[column sep = large]
& Y\ar[d] & &  \\
& F\ar[d] \ar[r,"u\times w "] & K(\Z/2,2n+3)\times K(\Z/2,2n+5)   \\
BU(n) \ar[r] \ar[ur, bend left = 5,dashed] \ar[uur, bend left = 20, dashed] & BU \ar[r,"c_{n+1}\times c_{n+2}"] & K(\Z,2n+2)\times K(\Z, 2n+4).
\end{tikzcd} 
\end{equation}

We apply \Cref{rmk:practice} to the fiber sequences 
\[Y \to F \to K(\Z/2,2n+3) \times K(\Z/2,2n+5)\] and 
\[ F \to BU \to K(\Z,2n+2) \times K(\Z,2n+4).\] 
Consider the action $m\: K(\Z,2n+1) \times K(\Z,2n+3) \times F \to F$, which induces a map of Serre spectral sequences. We find that
the number of homotopy lifts of a given class $h\: \CP^{n+2} \to BU$ to $Y$ equals the size of the cokernel of 
\[
\pi_1 (u\times w)_*: \pi_1(F^{\CP^{n+2}}, f) \to \pi_1( (K(\Z/2,2n+3)\times K(\Z/2, 2n+5))^{\CP^{n+2}},0 ),
\]
where $f\:\CP^{n+2} \to F$ is any lift of $h\: \CP^{n+2} \to BU$.
Since $\pi_1(BU^{\CP^{n+2}},h)=0$, by \Cref{lem:lift-2}, 
the image of $\pi_1(u \times w)_*$ can be computed as the set of all homotopy classes
\[(x,y,A)^*(m^*(u \times w)) - (0,0,A)^*(m^*(u\times w))\]
for $A \in \pi_1(F^{\CP^{n+2}}, f)$ fixed and $x,y$ varying over 
$\pi_1((K(\Z,2n+1)\times K(\Z, 2n+3))^{\CP^{n+2}},0)$, as summarized in
 \Cref{cd:action1mod4} below. 
\begin{equation}\label[diagram]{cd:action1mod4}
\begin{tikzcd}[column sep=.5em]
{S^1 \times \CP^{n+2}}\ar[r, dashed] \ar[dr, bend right = 5, "{(x, y, A)}\,\,\,\,\,\,\,\,\,\,\,\,\,\,\," {below, near start}]
& F \ar[r,"u\times w"] 
& {K(\Z/2,2n+3) \times K(\Z/2, 2n+5)} \\
& {K(\Z,2n+1)\times K(\Z,2n+3) \times F} \ar[u,"m"] & .
\end{tikzcd}
\end{equation}
Write $\iota_1t^i$ for the degree $2i+1$ generator of $\HZ{2}^*\Sigma \CP^{n+2}$. 
Then $x = x_0\cdot \iota_1t^n$ and $y = y_0\cdot \iota_1t^{n+1}$, where $x_0$ and $y_0$ are integers.

On the $E_2$-page of the Serre spectral sequence for $K(\Z,2n+1) \times K(\Z,2n+3) \to  F \to BU$, the cohomology class $u$ is represented by
$
\iota_{2n+3}\otimes 1 + \Sq^2\iota_{2n+1}\otimes 1 + \iota_{2n+1}\otimes c_1,
$
and the class $w$ is represented by 
$
\Sq^4\iota_{2n+1}\otimes 1 + \iota_{2n+1}\otimes c_2.
$
This implies that, on the $E_2$-page for the Serre spectral sequence of the fibration $K(\Z,2n+1) \times K(\Z,2n+3) \times F \to F \times F \to BU$, the elements $m^*u$ and $m^*w$ are represented as follows:
\begin{align*}
 m^*u &= [1\otimes u + \iota_{2n+3}\otimes 1 + \Sq^2\iota_{2n+1}\otimes 1 + \iota_{2n+1}\otimes c_1]\\
 m^*w &= [1 \otimes w  + \Sq^4\iota_{2n+1}\otimes 1 + \iota_{2n+1}\otimes c_2].
\end{align*}
It follows that
\begin{align*}
  & \, \left( (x,y,A)^*(m^*u)-(0,0,A)^*(m^*u), \,\, (x,y,A)^*(m^*w)-(0,0,A)^*(m^*w) \right) \\
  &=  \left( y_0\iota_1t^{n+1}+ \Sq^2(x_0\iota_1t^n) + x_0\iota_1t^n\cdot c_1(h)t, \,\,  \Sq^4(x_0\iota_1t^n) + c_2(h)t^2x_0\iota_1t^n \right) \\
  &= \left( \left(y_0+ nx_0+c_1(h)x_0\right)\iota_1t^{n+1}, \,\, (\binom{n}{2}x_0+ c_2(h)x_0 )\iota_1t^{n+2} \right).
\end{align*}
When $n\equiv 1\pmod{4}$, the above expression reduces to
 \[\left((y_0+x_0 + c_1(h)x_0)\iota_1t^{n+1}, \,\,c_2(h)x_0\iota_1t^{n+2}\right) \in \Z/2 \{t^{n+1}\} \times \Z/2 \{t^{n+2}\}.\]
For any $c_1(h)$, we can choose $x_0, y_0$ so that the first component above is nonzero. When $c_2(h)$ is even, the second component always vanishes. So, the cokernel is $\Z/2$ when $c_2(h)$ is even and is $0$ otherwise. We have proved:
\begin{prop} \label{thm:1mod4count-1}
Let $n\equiv 1 \mod 4$. The number of homotopy classes of lifts of $h\:\CP^{n+2}\to BU$ with $c_{n+1}(h)=c_{n+2}(h)=0$ to $Y$ equals $2$ when $c_2(h)$ is even, and equals $1$ otherwise.
\end{prop}
By \Cref{n1mod8-Y-enough}, for $n\equiv 1 \pmod{8}$, it is enough to compute lifts of $h$ to $Y$:
\begin{thm}\label{cor:1mod4count}
Suppose that $n\equiv 1 \mod 8$ and let $h\: \CP^{n+2} \to \c{BU}{}$ be given with $c_{n+1}(h)=c_{n+2}(h)=0$. The number of isomorphism classes of lifts of $h$ to $\c{BU(n)}{}$ is $2$ when $c_2(h)$ is even, and is $1$ otherwise.
\end{thm}

\Cref{thm:1mod4count-1} also implies:
\begin{cor}\label{cor:5mod4bound}

Suppose that $n\equiv 5 \mod 8$ and let $h\: \CP^{n+2} \to BU$ be such that $c_{n+1}(h)=c_{n+2}(h)=0$. The number of isomorphism classes of lifts of 

\[\c{h}{}\:\CP^{n+2} \xrightarrow{h}BU \to \c{BU}{}\]
to $\c{BU(n)}{}$ is at least $2$ when $c_2(h)$ is even, and is $1$ otherwise.
\end{cor}
\begin{proof} The fact that the number of lifts is at least $2$ when $c_2(h)\equiv 0\pmod{2}$ is immediate from \Cref{n1mod8-tildeL-enough} and \Cref{thm:1mod4count-1}. To see that we get only one lift when $c_2(h)$ is odd, fix a basepoint $f \in X^{\CP^{n+2}}$ lifting $h$. Such a lift exists since $n$ is odd, by \Cref{lem:exists-n-odd-c1-even}. Note that we have a homotopy commutative diagram as below:

\[\begin{tikzcd}\pi_1( X^{\CP^{n+2}},f) \ar[r,"\pi_1\tilde w_*"]\ar[dr,"\pi_1w_*\,\,\,\, " below] &\pi_1( K(\Z/4,2n+5)^{\CP^{n+2}},0) \ar[d] \ar[r,"\simeq"] &\Z/4 \ar[d,"q"]\\
 &\pi_1( K(\Z/2,2n+5)^{\CP^{n+2}},0) \ar[r,"\simeq"] & \Z/2, \end{tikzcd} \]
where $q$ is the quotient by $2\Z/4 \subset \Z/4$, and $X$ is as indicated in \Cref{diag:g1mod4}. Since $\pi_1w_*$ maps onto a generator for $\Z/2$ when $c_2(h)$ is odd, $\pi_1\tilde w_*$ must map onto a generator for $\Z/4$.
\end{proof}


%

\subsubsection{Reduction to $c_1\equiv 0 \pmod{4}$ and an alternate tower for $n\equiv 5\pmod{8}$}\label{n5mod8c1even}

Let $O(d)$ denote the line bundle on $\CP^{n+2}$ determined by $c_1(O(d))=d$. The next lemma is immediate from the fact that tensoring by $O(d)$ acts bijectively on $[\CP^{n+2},BU(n)]$, combined with basic facts about how tensoring by line bundles interacts with Chern classes.
\begin{lemma}\label{lb-trick-2} Let $h\: \CP^{n+2} \to BU$ have an unstable representative $V\: \CP^{n+2} \to BU(n)$. Let $h_d\: \CP^{n+2} \to BU$ denote the $K$-theory class of $V \otimes O(d)$. Then $\#\Vect^h_n(\CP^{n+2})=\#\Vect^{h_d}_n(\CP^{n+2})$. Moreover, $c_1(h_d)=c_1(h)+nd$ and $c_2(h_d)=c_2(h)+(n-1)c_1(h)d+\frac{n(n-1)}{2}d^2.$
\end{lemma} 
\begin{rmk}\label{lb-trick-2-rmk} Assuming $n\equiv 1 \pmod{4}$, we can always replace $h$ by another $K$-theory class $h_d$ where $c_1(h_d)\equiv 0\pmod{4}$. 
If $c_1(h)\equiv 1 \pmod{4}$ then we tensor by $O(3)$ and find that $c_1(h_3)\equiv 0 \pmod{4}$.
 If $c_1(h)\equiv 3\pmod{4}$, then we tensor by $O(1)$ and $c_1(h_1)\equiv 0\pmod{4}$. 
If $c_1(h)\equiv 2 \pmod{4}$, we tensor by $O(2)$ and find $c_1(h_2)\equiv 0 \pmod{4}$.
So, we can translate the computation to a similar one with the hypothesis that $c_1(h)\equiv 0\pmod{4}$.
\end{rmk}

\begin{rmk} The above indicates why there is no dependence on $c_1\pmod{4}$ in the case $n\equiv 1 \pmod{4}$: the value $c_1\pmod{4}$ is not invariant under certain constructions that preserve the count of interest.\end{rmk}

With this lemma in hand, we construct a tower of fibrations classifying vector bundles with even first Chern class. The tower of fibrations we use is summarized in \Cref{diag:g1mod4-2-2}. 
\begin{defn}\label{def-mod2-c1} Let $c_1(2)$ denote the composite $c_1 \pmod{2}\: BU \xrightarrow{c_1}K(\Z,2) \xrightarrow{q} K(\Z/2,2),$ where $q$ is induced by the surjective map $\Z \to \Z/2$. Let
\begin{align*} BU_2&=\fib\left(c_1(2)\:BU\to K(\Z/2,2)\right),\\ 
BU(n)_2&=\fib \left(c_1(2)\:BU(n)\to K(\Z/2,2)\right),
\quad \text{ and }\\
F_2&=\fib\left(c_{n+1}\times c_{n+2}\: BU_2 \to K(\Z,2n+2) \times K(\Z,2n+4)\right).
\end{align*}
\end{defn}

\begin{rmk}\label{coh-mod2-c1} Note that $BU(n)_2$ and $BU_2$ have Chern classes pulled back from those in the cohomology of $BU$. Moreover,
 \[\HZ{2}^*BU_2 \cong \Z/2[ \bar{c}_1,c_2, c_3,\ldots ],\quad \text{ while } \quad \HZ{2}^*BU(n)_2 \cong \Z/2[ \bar{c}_1,c_2, c_3,\ldots,c_n ].\] 
 In the above, $\bar{c}_1$ is detected by $\iota_1^2$ in the Serre spectral sequence for the fibration
\[K(\Z/2,1) \to BU_2 \to BU,\]
(here, $\iota_1$ generates the degree $1$ cohomology of the fiber). \end{rmk}
The relevant Steenrod square actions on $\HZ{2}^*BU_2$ are
\[
\Sq^2c_{n+1} = c_{n+2}, \quad \Sq^2c_{n+2} = 0, \quad \text{and} \quad \Sq^4c_{n+1} = c_2c_{n+1},
\]
by \Cref{prop:Chernform2} combined with the fact that $c_1\equiv 0 \pmod{2}$ for the universal bundle on $BU_2$. 
We first note the following lemma, following from the fact that $\CP^{n+2}$ has even cells:
\begin{lemma}\label{rmk:2ok} Given $h\:\CP^{n+2} \to BU$ with $c_1(h)$ even, $h$ lifts uniquely to $BU_2$. Similarly, for any $r$, given $V\: \CP^{n+2} \to BU(r)$ with $c_1(V)$ even, $V$ lifts uniquely to the fiber of $c_1\pmod{2}\: BU(r) \to K(\Z/2,2).$\end{lemma}

\begin{rmk}
The previous lemma ensures that, when $c_1(h)$ is even, we may replace enumeration of lifts of a given map
\[
\CP^{n+2}\xrightarrow{h} BU
\]
along $BU(n)\to BU$ with enumeration of lifts of the unique lift
\[
\CP^{n+2}\xrightarrow{h} BU_2
\]
along $BU(n)_2\to BU_2$. The advantage of this replacement is that the simplified relation
$\Sq^2c_{n+1}=c_{n+2}$ in $\HZ{2}^*BU_2$ permits a delooping needed to get a useful action, resolving the issue depicted in \Cref{issue:1mod2} and discussed below that diagram.
\end{rmk}

To enumerate lifts of $h\:\CP^{n+2} \to BU_2$ to $BU(n)_2$, we first consider lifts to $F_2$. We form the Serre spectral sequence for the fibration
\[ K(\Z,2n+1) \times K(\Z,2n+3)\to F_2 \to BU_2,\]
where $F_2$ is the fiber of $c_{n+1} \times c_{n+2}$. 
 We find that
\begin{itemize}
 \item $d_{2n+4}(\Sq^2\iota_{2n+1}) = \Sq^2d_{2n+2}(\iota_{2n+1}) = \Sq^2c_{n+1} = c_{n+2}$,
 \item $d_{2n+5}(\Sq^3\iota_{2n+1}) = \Sq^3d_{2n+2}(\iota_{2n+1}) = \Sq^3c_{n+1} = 0$,
 \item $d_{2n+6}(\Sq^2\iota_{2n+3}) = \Sq^2d_{2n+4}(\iota_{2n+3}) = \Sq^2c_{n+2} = 0$, and
 \item $d_{2n+6}(\Sq^4\iota_{2n+1}) = \Sq^4d_{2n+2}(\iota_{2n+1}) = \Sq^4c_{n+1} = c_2c_{n+1}$.
\end{itemize}
It follows that 
\[
\Sq^2\iota_{2n+1}+\iota_{2n+3}, \quad \Sq^3\iota_{2n+1}, \quad \Sq^2\iota_{2n+3}, \quad \text{and} \quad \Sq^4\iota_{2n+1},
\]
are permanent cycles that detect classes $u,v,m,w$, respectively, in $\HZ{2}^{*}F_2$. These classes are the pullbacks of the classes of the same name in $\HZ{2}^*F$ defined in \Cref{n1mod8-start}, justifying the abuse of notation. They are represented in the double complex as follows:
\begin{align*}
u &=[\Sq^2\iota_{2n+1}\otimes 1  + \iota_{2n+3}\otimes 1] & v= [\Sq^3\iota_{2n+1}\otimes 1] \\
m&= [\Sq^2\iota_{2n+3}\otimes 1] & w=[\Sq^4\iota_{2n+1} \otimes 1  +\iota_{2n+1} \otimes c_2],
\end{align*} 
and, up to terms of lower filtration,
\[
 \Sq^1u=v,  \quad  \Sq^2u=m, \quad \Sq^1w=\Sq^2v, \quad \text{and} \quad \Sq^3u=\Sq^1m.
\]

We now build the next stage of the tower. Consider the map
\[
u \: F_2 \to K(\Z/2, 2n+3).
\]
Write $X_2$ for its homotopy fiber, which fits into \Cref{diag:g1mod4-00}.
\begin{equation}\label[diagram]{diag:g1mod4-00}
\begin{tikzcd}[column sep = large]
& X_2\ar[d] & &  \\
& F_2\ar[d] \ar[r,"u "] & K(\Z/2,2n+3)   \\
BU(n)_2\ar[r] \ar[ur, bend left = 5,dashed] \ar[uur, bend left = 20, dashed] & BU_2 \ar[r,"c_{n+1}\times c_{n+2}"] & K(\Z,2n+2)\times K(\Z, 2n+4).
\end{tikzcd} 
\end{equation}
Consider the Serre spectral sequence for $X_2$ as shown in \Cref{fig4-00}.

\begin{figure}[h]
\centering
\textbf{The $E_2$-page of the Serre spectral sequence computing $\HZ{2}^*X_2$}\par\medskip
\begin{tikzpicture}
\matrix (m) [matrix of math nodes,
nodes in empty cells,nodes={minimum width=2.5ex,
minimum height=2.5ex,outer sep=-1pt},
column sep=.25ex,row sep=.25ex]{
t&\\
2n+4 && \Sq^2\iota_{2n+2} &  & & & & & & & & & &&& \\
2n+3 && \Sq^1\iota_{2n+2} & & & & & & & & & & &&& \\
2n+2 && \iota_{2n+2} & &  & & & & & & & & &&& \\
\vdots\\
0&&1 & & & & \hdots &&  u &  & v & &m,w & \\
&& & & & & & & & & & & && \\
\quad\strut & & 0  & &  && \hdots& & 2n+3 & &2n+4& & 2n+5 & & & q \strut \\};
\draw[->] (m-4-3.south east) -- (m-6-9.north west);
\draw[->] (m-3-3.south east) -- (m-6-11.north west);
\draw[->] (m-2-3.south east) -- (m-6-13.north west);
\draw[thick] (m-1-2) -- (m-8-2) ;
\draw[thick] (m-7-1) -- (m-7-15) ;
\end{tikzpicture}
\caption{A schematic of the $E_2$-page for the $\Z/2$-cohomology Serre spectral sequence for the fiber sequence $K(\Z/2,2n+2) \to X_2 \to F_2$. Only key multiplicative generators are shown. Arrows indicate differentials.}\label{fig4-00}
\end{figure}
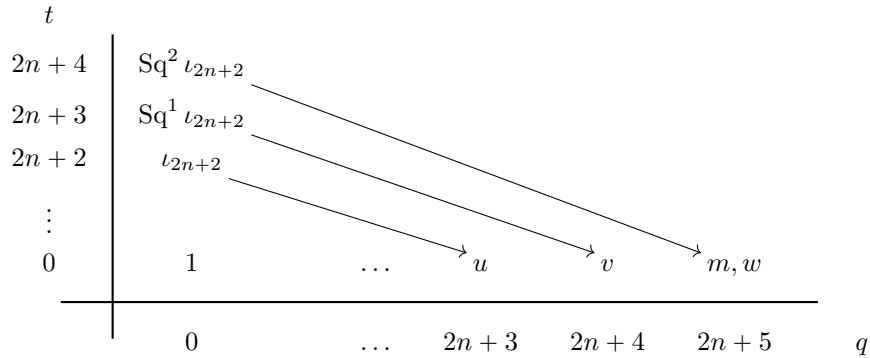

Note that the classes $u,v,m,\Sq^1m$ are all killed by transgressions of Steenrod squares on $\iota_{2n+2}$. We find that 
\[\HZ{2}^{*\leq 2n+6}X_2\cong \left( \Z/2[\bar{c}_1,c_2,\ldots,c_n]\otimes \Z/2\{1, w\} \right)^{*\leq 2n+6},\]
where $w$ is in degree $2n+5$ and is pulled back from the cohomology of $F_2$. Moreover, since $\pi_{2n+4}BU(n)\cong \Z/4$, by \Cref{lem:lift-pin} $w$ lifts to a $\Z/4$-cohomology class.
Let $\tilde{w}\: X_2 \to K(\Z/4,2n+5)$ be a lift of $w$ to $\Z/4$-cohomology. We define \[\tilde{Y}_2=\fib\left( \tilde{w}\: X_2 \to K(\Z/4,2n+5) \right).\] 
There is a map $BU(n)_2\to \tilde{Y}_2$ inducing an isomorphism
 \[\HZ{2}^{*\leq 2n+5}\tilde{Y}_2\cong \HZ{2}^{*\leq 2n+5}BU(n)_2.\] 
So, we conclude:
\begin{lemma}\label{n1mod8-tildeL-enough-2} If $n\equiv 5 \pmod{8}$, the natural map $BU(n)_2 \to \tilde{Y}_2$ induces a bijection \[[\CP^{n+2}, \c{(BU(n)_2)}{}]\cong [\CP^{n+2},\c{(\tilde{Y}_2)}{}].\]
\end{lemma}
All in all, we obtain the tower of principal fibrations depicted in \Cref{diag:g1mod4-2-2}.
\begin{equation}\label[diagram]{diag:g1mod4-2-2}
\begin{tikzcd}[column sep = large]
&\tilde{Y}_2\ar[d]\\
& X_2\ar[d]\ar[r,"\tilde{w}"] & K(\Z/4,2n+5)  \\
& F_2\ar[d] \ar[r,"u "] & K(\Z/2,2n+3) \\
&BU_2 \ar[r,"c_{n+1} \times c_{n+2}"] &  K(\Z,2n+2)\times K(\Z, 2n+4).
\end{tikzcd} 
\end{equation}
\begin{lemma}\label{nothing-interesting} Suppose that $h\: \CP^{n+2} \to \c{(BU_2)}{}$ admits a lift to $\c{(BU(n)_2)}{}$. Among the homotopy classes of lifts of $h\: \CP^{n+2} \to \c{(BU_2)}{}$ to $X_2,$ exactly one homotopy class lifts further to $\c{(\tilde{Y}_2)}{}$.\end{lemma}
\begin{proof}
Examining \Cref{diag:g1mod4-2-2}, we find that 
\begin{equation}\label{eq-new-nice}[\CP^{n+2},X_2]\cong [\CP^{n+1},X_2],\end{equation} induced by restriction along the inclusion of the $(2n+2)$-skeleton of $\CP^{n+2}$.  
Any two rank $n$ vector bundles on $\CP^{n+2}$ with the same $K$-theory class have the same restriction to $\CP^{n+1}$ (\Cref{lem1}).
 By \Cref{n1mod8-tildeL-enough-2} and \Cref{rmk:2ok}, any two maps $\CP^{n+2} \to \tilde{Y}_2$ lifting the same map to $BU_2$ have the same restriction to $\CP^{n+1}$. Using \Cref{eq-new-nice}, any two lifts of $h$ to $\tilde Y_2$ compose to the same homotopy class along $\tilde Y_2 \to X_2$.
\end{proof}
We have a homotopy commutative diagram as depicted in \Cref{diag:g1mod4-2-3}.
\begin{equation}\label[diagram]{diag:g1mod4-2-3}
\begin{tikzcd}[column sep = 5em,row sep=1.5em]
\fib\Sqp\ar[d] \ar[r]& X_2\ar[d]\ar[r] &*\ar[d] \\
K(\Z,2n+1)\times K(\Z,2n+3)\ar[d]\ar[r]& F_2\ar[d] \ar[r,"u "] & K(\Z/2,2n+3) \ar[from=ll,crossing over,bend right=8,"\Sq^2+\iota_{2n+3}" {below, near end}]  \\
* \ar[r]& BU_2 \ar[r,"c_{n+1}\times c_{n+2}" below] & K(\Z,2n+2)\times K(\Z,2n+4),
\end{tikzcd} 
\end{equation}
where we defined $\fib\Sqp$ as the fiber of \[\Sqp:=\Sq^2+\iota_{2n+3}:K(\Z,2n+1) \times K(\Z,2n+3) \to K(\Z/2,2n+3).\] In \Cref{diag:g1mod4-2-3}, it is immediate that the upper right-hand square is a pullback, as is the outer upper horizontal rectangle. Therefore, the upper left-hand square is also a pullback. Since $K(\Z,2n+1) \times K(\Z,2n+3) \to F_2 \to BU_2$ is also a fiber sequence, the outer left-hand vertical rectangle is a pullback as well, and we obtain a fiber sequence 
\[ \fib\Sqp \to X_2 \to BU_2.\]
Moreover, we will show this fiber sequence deloops. Consider \Cref{eq:fib-tildeG-BU3-2-0} below.

\begin{equation}\label[diagram]{eq:fib-tildeG-BU3-2-0}
\begin{tikzcd}
\fib\Sqp\ar[d] \ar[r] &K(\Z,2n+1) \times K(\Z,2n+3)\ar[d]\\
X_2\ar[d] \ar[r] & F_2\ar[d] \\
BU_2 \ar[d,dashed] \ar[r] & BU_2\ar[d,"c_{n+1}\times c_{n+2} "] \\
\fib_{2n+2}(\Sq^2+\iota_{2n+4})\ar[r]& K(\Z,2n+2)\times K(\Z,2n+4) \ar[r,"\Sqp" below] & K(\Z/2,2n+4).
\end{tikzcd}
\end{equation}
The dashed arrow exists since $\Sq^2c_{n+1}=c_{n+2}$ in the cohomology of $BU_2$.
Thus, we get an action 
\[\tilde m\: \fib\Sqp \times X_2 \to X_2,\]
which is transitive after applying $\pi_1((-)^{\CP^{n+2}},\star)$ by \Cref{lem:lift-22}.

Combining the previous two observations with \Cref{cor:coker-1} and \Cref{lem:lift-2}, we deduce:
\begin{cor}\label{cor:computation-needed}Let $h\: \CP^{n+2} \to BU_2 \to \c{(BU_2)}{}$  admit a lift to $\c{(BU(n)_2)}{}$. Let $\tilde h$ be any lift of $h$ to $X_2$ that lifts further to $\tilde{Y}_2$. Then lifts of $h$ to $\c{(BU(n)_2)}{}$ are in bijection with the quotient of $\pi_1(K(\Z/4,2n+5)^{\CP^{n+2}},0)$ by the subgroup generated by elements of the form 
\[  (x,A)^* (\tilde m^* \tilde w)- (0,A)^*(\tilde m^* \tilde w)\]
where $A\in \pi_1(X_2^{\CP^{n+2}},\tilde{h})$ is fixed, and $x$ ranges over $[\Sigma \CP^{n+2},\fib \Sqp]$.
\end{cor}
To actually apply \Cref{cor:computation-needed}, we need some $\Z/4$-cohomology computations.
These are in \Cref{app:Z4-cohomology}.

\subsubsection{The action $\tilde m$ when $n\equiv 5 \pmod{8}$}\label{n5mod8-action1}
We have a commutative diagram relating various actions on relevant cohomology classes.
\[\begin{tikzcd}[row sep=.7em]
\fib\Sqp \times X_2 \ar[r,"\tilde m"] \ar[d,"="]& X_2\ar[d,"="] \ar[r,"\tilde w"] &K(\Z/4,2n+5)\ar[d,"q"]\\
\fib\Sqp \times X_2 \ar[r,"\tilde m"] \ar[d]& X_2 \ar[r," w"] \ar[d]&K(\Z/2,2n+5)\ar[d,"="]\\
K(\Z,2n+1)\times K(\Z,2n+3) \times F_2 \ar[r," m"] & F_2 \ar[r," w"] &K(\Z/2,2n+5),
\end{tikzcd}
\]
from which it follows that $\tilde m^* \tilde w$ is a $\Z/4$-cohomology lift of the $\Z/2$-cohomology class $m^*w$. 
Consider the diagram of cohomology coefficients:
\begin{equation}\label[diagram]{vary-coeffs}0\to\Z/2\xrightarrow{i} \Z/4 \xrightarrow{p} \Z/2\to 0,\end{equation} 
which, for any $X$, gives an induced map on cohomology 
\begin{equation}\label{ip}\HZ{2}^*X \xrightarrow{i_*} \HZ{4}^*X\xrightarrow{p_*}\HZ{2}^*X\end{equation} fitting into a long exact sequence whose connecting homomorphism is, by definition, $\Sq^1$. 
With notation as in \Cref{coh:Z2-Sqp},
\[p_*\tilde m^* \tilde w =1 \otimes w+ \Sq^4\kappa_{2n+1} \otimes 1 + \kappa_{2n+1} \otimes c_2.\]

Consider now the K\"{u}nneth map
\[ j\:\HZ{4}^{*}\fib\Sqp \otimes \HZ{4}^*X_2\to\HZ{4}^{*}(\fib\Sqp \times X_2).\] 
\begin{rmk}\label{strat:no-kunneth} We would like to compute $\tilde m^* \tilde w$ in the tensor product of $\Z/4$-cohomology groups, but the K\"{u}nneth spectral sequence for the product $\fib\Sqp\times X_2$ does not collapse.
Our next few steps circumvent this: first, we write down a candidate for $\tilde m^* \tilde w$ in the tensor product of the cohomologies, whose image under $j$ is a good approximation to $\tilde m^* \tilde w$; second, we show how to compute with this better-behaved class instead. \end{rmk}

We define \[R:= 1\otimes \tilde{w}+\tilde{\Sq}^4 \tilde \kappa_{2n+1} \otimes  1 + \tilde \kappa_{2n+1}\otimes c_2 \] in $ \HZ{4}^{*}\fib\Sqp \otimes \HZ{4}^* X_2,$ where $\tilde{\Sq}^4 \tilde \kappa_{2n+1}$ is a $\Z/4$-cohomology lift of $\Sq^4\kappa_{2n+1}$ (see \Cref{coh:Z2-Sqp-1}).
Note that $jR$ is a lift of $\tilde m^*w$ to $\Z/4$-cohomology. 
\begin{lemma}\label{lem-A-not-tildeW-0} $\tilde m^* \tilde w -jR$ is in the image of the change-of-coefficient map 
\[i_*\: \HZ{2}^{2n+5}(\fib\Sqp \times X_2) \to \HZ{4}^{2n+5}(\fib\Sqp\times X_2).  \]

\end{lemma}
\begin{proof} This follows from the fact that $p_*\tilde m^* \tilde w = p_*jR= \tilde m^*w.$
\end{proof}

Thus, $ \tilde m^*\tilde w-jR= i_*Z $
where $Z \in \HZ{2}^*\fib\Sqp \otimes \HZ{2}^*X_2\cong \HZ{2}^*(\fib\Sqp \times X_2).$
Considering the $\Z/2$-cohomology of $\fib\Sqp$ and $ X_2$ and applying the K\"{u}nneth isomorphism for $\Z/2$-cohomology, we can write:
\begin{equation}\label{guess-z-0}Z = 1\otimes Z_0+\epsilon_0\kappa_{2n+1} \otimes \bar{c}_1^2 +\epsilon_1\kappa_{2n+1} \otimes c_2+\epsilon_2\Sq^2\kappa_{2n+1} \otimes \bar{c}_1+ \epsilon_3\Sq^4 \kappa_{2n+1} \otimes 1, \end{equation}
for some choices of $Z_0 \in \HZ{2}^* X_2$ and $\epsilon_i \in \Z/2$ (we refer to \Cref{coh:Z2-Sqp-1} for notation).

\begin{rmk} Examining \Cref{guess-z-0}, we see that at least one factor in each summand lifts to a $\Z/4$-cohomology class. This implies that $\tilde m^*\tilde w -jR$ is in fact in the image of the map
\[ j\: \HZ{4}^*\fib\Sqp\otimes \HZ{4}^*  X_2 \to \HZ{4}^*(\fib\Sqp\times  X_2).\]
\end{rmk}

\begin{thm}\label{thm:n5mod8-2local} Suppose that $n\equiv 5 \pmod{8}$ and $h\: \CP^{n+2} \to BU \to  \c{BU}{}$ is given with $c_{n+1}(h)=c_{n+2}(h)=0$. The number of homotopy classes of lifts of $h$ to $\c{BU(n)}{}$ is equal to:
\begin{itemize}
\item $4$ if $c_1(h)$ is even and $c_2(h)$ is divisible by $4$, or if $c_1(h)$ is odd and $c_2(h)\equiv 2 \pmod{4}$;
\item $2$ if $c_1(h)$ is odd and $c_2(h)\equiv 0 \pmod{4}$, or if $c_1(h)$ is even and $c_2(h)\equiv 2\pmod{4}$; and
\item $1$ if $c_2(h)$ is odd.
\end{itemize}
\end{thm}
\begin{proof} We have already dealt with the case in which $c_2(h)$ is odd in \Cref{cor:5mod4bound}, so we assume $c_2(h)$ is even. Additionally, we first assume $c_1(h) \equiv 0\pmod{4}$. Otherwise, we appeal to \Cref{lb-trick-2-rmk} to reduce to this case.

Let a lift $\tilde{h}\: \CP^{n+2} \to  X_2$ be given, and write $y\: S^1\times \CP^{n+2} \to  X_2$ for $0\times \tilde{h}$, which represents a fixed element in $\pi_1(X_2^{\CP^{n+2}},\tilde{h})$. 
Applying \Cref{cor:computation-needed}, we wish to compute
\begin{equation}\label{eq:diff-0}
(0,y)^*(\tilde m ^* \tilde w) - (x,y)^*(\tilde m ^* \tilde w)=(0,y)^*(jR+i_*Z)-(x,y)^*(jR+i_*Z)
\end{equation} 
as $x$ ranges over $[\Sigma\CP^{n+2},\fib\Sqp]$. 
We need to compute the operation 
\[\Sqp=\Sq^2+\iota_{2n+3}\: H_{\Z}^{2n+1}\Sigma\CP^{n+2} \times H_{\Z}^{2n+3}\Sigma\CP^{n+2} \to \HZ{2}^{2n+3}\Sigma\CP^{n+2} .\]
Write $\iota_1t^i$ for a degree $2i+1$ generator of $H_{\Z}^*\Sigma \CP^{n+2}$. Since $n$ is odd, 
\[(\Sq^2+\iota_{2n+3})(a_0\iota_1t^{n}+b_0\iota_1t^{n+1})=(a_0+b_0)\iota_1t^{n+1}\pmod{2}.\]
We can therefore identify $[\Sigma\CP^{n+2},\fib\Sqp]$ with cohomology classes $a_0\iota_1t^{n}+b_0\iota_1t^{n+1}$ where $a_0+b_0\equiv 0\pmod{2}$. We define
\[a=a_0\iota_1t^{n}, \quad b=b_0\iota_1t^{n+1}.\]

We break the computation of \Cref{eq:diff-0} into two steps. First, consider the diagram:
\begin{equation}\label[diagram]{ijstuff-0}
\begin{tikzcd}
\Sigma \CP^{n+2} \ar[r,bend right,"{(a,b,y)}" below] \ar[r,bend left,"{(0,0,y)}" above] &
{ \fib\Sqp \times X_2}\ar[r,"Z"] &
 K(\Z/2,2n+5) \ar[r,"i"]&
 K(\Z/4,2n+5).\end{tikzcd}
\end{equation}
 Referring to \Cref{guess-z-0}, we find that:
\begin{align*} (0,0,y)^*Z- (a,b,y)^*Z
&=   1\otimes Z_0-1 \otimes Z_0-\epsilon_0(a,b)^*\kappa_{2n+1} \otimes y^*\bar{c}_1^2 -\epsilon_1(a,b)^*\kappa_{2n+1} \otimes y^*c_2\\
& \quad \quad -  \epsilon_2(a,b)^*\Sq^2\kappa_{2n+1} \otimes y^*\bar{c}_1-\epsilon_3(a,b)^*\Sq^4 \kappa_{2n+1} \otimes 1\\
&=a_0\left( \epsilon_0\bar c_1(h)^2 + \epsilon_1 c_2(h)+\epsilon_2 \bar c_1(h)\right)\iota_1t^{n+2},
\end{align*}
where we use that $\Sq^4t^n=0$ if $n\equiv 1 \pmod{4}$. Moreover, $c_1(h)\equiv 0\pmod{4}$ implies that  $\bar{c}_1(h)=0\pmod{2}$. The above reduces to:
\begin{align*} (0,0,y)^*Z- (a,b,y)^*Z &=a_0\epsilon_1c_2(h)\iota_1t^{n+2},
\end{align*}
and
\begin{align*} (0,0,y)^*i_*Z- (a,b,y)^*i_*Z &=2a_0\epsilon_1c_2(h)\iota_1t^{n+2}.
\end{align*}
Next we study $jR.$ 
\begin{align*}(0,0,y)^*(jR) - (a,b,y)^*(jR)&=a_0\iota_1\tilde{Sq}^4t^n+ a_0c_2(h)\iota_1t^{n+2}.\end{align*}
Since $\tilde m^* \tilde w= jR+i_*Z,$ we deduce:
\begin{align*}(0,0,y)^*(\tilde m ^* \tilde w) - (a,b,y)^*(\tilde m ^* \tilde w)&=a_0\iota_1\tilde{Sq}^4t^n + a_0 c_2(h)\iota_1t^{n+2} + 2a_0\epsilon_1c_2(h)\iota_1t^{n+2}.\end{align*}
When $h=0$, Hu's enumeration \Cref{thm:yang-main} shows that the cokernel is $\Z/4$, forcing $\tilde{Sq}^4t^n=0$. This shows that the cokernel is $\Z/4$ whenever $c_2(h)\equiv 0 \pmod{4}$.
If $c_2(h) \equiv 2 \pmod{4}$, 
$(0,0,y)^*(\tilde m ^* \tilde w) - (a,b,y)^*(\tilde m ^* \tilde w)$ has image $2\Z/4$ as $a,b$ vary. So the cokernel is $\Z/2$. If $c_2(h)$ is odd, the cokernel is zero.

Lastly, we do some bookkeeping for different values of $c_1(h)\pmod{4}$. By the $2$-complete version of
\Cref{lb-trick-2} and \Cref{lb-trick-2-rmk}:
\begin{itemize}
\item Case 1: $c_1(h)\equiv 1 \pmod{4}$. In this case, $\#\c{\Vect^h_n(\CP^{n+2})}{}=\# \c{\Vect^{h_3}_{n}(\CP^{n+2})}{},$ where $c_1(h_3)\equiv 0 \pmod{4}$ and $c_2(h_3)\equiv c_2(h)+\frac{n(n-1)}{2}\pmod{4}.$ In particular, if $c_1(h)\equiv 1 \pmod{4}$ and $n\equiv 5\pmod{8}$, then $c_2(h_3)\equiv c_2(h)+2\pmod{4}$.
\item Case 2: $c_1(h)\equiv 2\pmod{4}$. In this case, $\#\c{\Vect^h_n(\CP^{n+2})}{}=\#\c{\Vect^{h_2}_{n}(\CP^{n+2})}{},$ where $c_1(h_2)\equiv 0 \pmod{4}$ and $c_2(h_2)\equiv c_2(h)+4\frac{n(n-1)}{2}\pmod{4}.$ In particular, if $c_1(h)\equiv 2 \pmod{4}$ and $n\equiv 5\pmod{8}$, then $c_2(h_2)\equiv c_2(h)\pmod{4}$.
\item Case 3: $c_1(h)\equiv 3\pmod{4}$. In this case, $\#\c{\Vect^h_n(\CP^{n+2})}{}=\#\c{\Vect^{h_1}_{n}(\CP^{n+2})}{},$ where $c_1(h_1)\equiv 0 \pmod{4}$ and $c_2(h_1)\equiv c_2(h)+\frac{n(n-1)}{2}\pmod{4}.$ In particular, if $c_1(h)\equiv 3 \pmod{4}$ and $n\equiv 5\pmod{8}$, then $c_2(h_1)\equiv c_2(h)+2\pmod{4}$.
\end{itemize}
We see that, if $c_1(h)$ is even, then the number of lifts of $h$ to $\c{BU(n)}{}$ is the greatest common divisor of $c_2(h)$ and $4$. If $c_1(h)$ is odd but $c_2(h)$ is even, then the number of lifts of $h$ to $\c{BU(n)}{}$ is the greatest common divisor of $c_2(h)+2$ and $4$.
If $c_2(h)$ is odd, there is a unique lift.
\end{proof}

\appendix

\section{$\Z/4$-cohomology calculations}\label{app:Z4-cohomology}
The $2$-complete enumeration problems when $n\equiv 5\pmod{8}$ require some $\Z/4$-cohomology computations. 
Recall the short exact sequence of coefficient modules from \Cref{vary-coeffs}
\[0\to\Z/2\xrightarrow{i} \Z/4 \xrightarrow{p} \Z/2\to 0.\] 
The induced map on cohomology 
\begin{equation}\label{ip-0}\HZ{2}^*K(\Z,2n) \xrightarrow{i_*} \HZ{4}^*K(\Z,2n)\xrightarrow{p_*}\HZ{2}^*K(\Z,2n)\end{equation} fits into a long exact sequence whose connecting homomorphism is, by definition, $\Sq^1$. Considering the action of $\Sq^1$ on $\HZ{2}^*K(\Z,2n)$, we deduce:
\begin{lemma}\label{lem:coh-2n1-4} For $n\geq 3$ and $j\geq 2n+1$, let $\iota_{j}$ generate $\HZ{2}^{j}K(\Z,j)$.
\begin{align*} 
\HZ{4}^{j}K(\Z,j)&\cong \Z/4\{ \tilde\iota_{j}\} & \text{ where } p_*\tilde\iota_{j}=\iota_{j}, \, i_*\iota_{j}=2\tilde\iota_{j}\\
\HZ{4}^{j+1}K(\Z,j)&=0  & \\
\HZ{4}^{j+2}K(\Z,j) & \cong \Z/2 \{ i_*\Sq^2\iota_{j} \}& \\
\HZ{4}^{j+3}K(\Z,j) &\cong \Z/2 \{ {x_{j+3}} \}& \text{ where } p_*x_{j+3}=\Sq^3\iota_{j} \\
\HZ{4}^{j+4}K(\Z,j)& \cong \Z/2\{ i_*\Sq^4\iota_{j} \}& 
\end{align*}
\end{lemma}
To apply \Cref{cor:computation-needed} in the case $n\equiv 5 \pmod{8}$, we need to compute $\Z/4$-cohomology of 
\[\fib\Sqp:=\fib\left(\Sq^2+\iota_{2n+3}\: K(\Z,2n+1) \times K(\Z,2n+3) \to K(\Z/2,2n+3) \right).\]
We first compute the $\Z/2$-cohomology. From the Serre spectral sequence for the $\Z/2$-cohomology of the fibration 
\[K(\Z/2,2n+2) \to \fib\Sqp \to K(\Z,2n+1) \times K(\Z,2n+3),\]
we find that
\begin{lemma}\label{coh:Z2-Sqp} With notation as above, 
\[ \HZ{2}^{*\leq 2n+5} \fib \Sqp \cong\Z/2\{1, \kappa_{2n+1}, \, \Sq^2\kappa_{2n+1},\, \Sq^4\kappa_{2n+1}\},\]
where $\kappa_{2n+1}$ in degree $2n+1$ is the pullback of the generator $\iota_{2n+1} \in \HZ{2}^{2n+1}K(\Z,2n+1).$ The Steenrod squares are as indicated, with $\Sq^1\kappa_{2n+1}=0, \, \Sq^3\kappa_{2n+1}=0,\, \Sq^5\kappa_{2n+1}=0$.
\end{lemma}
Using \Cref{ip-0}, we deduce:
\begin{lemma}\label{coh:Z2-Sqp-1} The $\Z/4$-cohomology of $\fib\Sqp$ in degrees at most $2n+2$ is given by:
 \[
\HZ{4}^*\fib\Sqp \cong  \begin{cases} 0 ,& 0<*<2n+1,\, \text{ or } *=2n+2,\, \\
\Z/4 \{ \tilde\kappa_{2n+1}\}, & *=2n+1.\end{cases}
\]
Additionally, $\HZ{4}^{2n+4}\fib\Sqp=0$, while there are short exact sequences
\[ 0 \to \HZ{2}^{2n+3} \fib\Sqp \xrightarrow{i_*} \HZ{4}^{2n+3} \fib\Sqp \xrightarrow{p_*} \HZ{2}^{2n+3}\fib\Sqp \to 0,\]
and
\[ 0 \to \HZ{2}^{2n+5} \fib\Sqp \xrightarrow{i_*} \HZ{4}^{2n+5} \fib\Sqp \xrightarrow{p_*} \HZ{2}^{2n+5}\fib\Sqp \to 0.\]
\end{lemma}

\section{Actions of Steenrod powers on Chern classes}\label{app:Chernformulas}

We prove the formulas needed for the action of $P^1$ on $c_n$ when $p=3$, and the actions of $\Sq^2$ and $\Sq^4$ on $c_n$ when $p=2$.

\begin{prop} \label{prop:Chernform3}
When $p=3$, the action of the Steenrod power $P^1$ on the $n$-th Chern class $c_n$ is given by the following formula:
\begin{equation}
P^1 c_n = c_1^2c_n + c_2c_n +2c_1c_{n+1} + (n+2)c_{n+2}.
\end{equation}
\end{prop}

\begin{proof}
We use the splitting principle. Consider the standard map $BU(1)^{\times (n+2)} \to BU(n+2)$, which, in $\Z/3$-cohomology, sends $c_n$ to $\sigma_n(x_1, \cdots, x_{n+2})$ (the $n$-th elementary symmetric polynomial in variables $x_1, \cdots, x_{n+2}$). It follows that
\begin{align*}
P^1\sigma_n(x_1, \cdots, x_{n+2}) &= P^1(\sum_{\substack{i< j}} \prod_{r\neq i,j} x_r)\,\,\, = \,\,\, \sum_{i< j} (\prod_{r\neq i,j} x_r)(\sum_{\substack{k\neq i,j}} x_k^2) \\
&= (\sum_{i < j}\prod_{r\neq i,j} x_r)( \sum_{k} x_k^2)
-\left((\sum_i\prod_{k \neq i} x_k) (\sum_r x_r) -(n+2)\prod_{l}x_l\right) \\
&=(\sum_{i < j}\prod_{r\neq i,j} x_r)\left( (\sum_{k} x_k)^2-2\sum_{i<j}x_ix_j \right)
-\left((\sum_i\prod_{k \neq i} (x_k)) (\sum_r x_r) -(n+2)\prod_{l}x_l\right) \\
 &= \sigma_1^2\sigma_n - 2\sigma_2\sigma_n - \sigma_1\sigma_{n+1} + (n+2)\sigma_{n+2},
\end{align*}
and hence $
P^1c_n = c_1^2c_n + c_2c_n + 2c_1c_{n+1} + (n+2)c_{n+2}.$
\end{proof}

\begin{prop} \label{prop:Chernform2}
When $p=2$, the actions of the Steenrod squares $\Sq^2$ and $\Sq^4$ on $c_n$ are
\begin{equation} \label{eq:sq2}
\Sq^2 c_n = c_1c_n + (n+1)c_{n+1},
\end{equation}
and
\begin{equation} \label{eq:sq4}
\Sq^4 c_n = c_2c_n + nc_1c_{n+1} + \frac{(n+2)(n-1)}{2}c_{n+2}.
\end{equation}
\end{prop}

\begin{proof}

Again, we use the splitting principle. Consider the map $BU(1)^{\times (n+2)} \to BU(n+2)$ that, in $\Z/2$-cohomology, sends $c_n$ to $\sigma_n=\sigma_n(x_1, \cdots, x_{n+2}),$ the $n$-th elementary symmetric polynomial in $x_1,\ldots, x_{n+2}$. We find that

\begin{align*}
\Sq^4\sigma_n(x_1, \cdots, x_{n+2}) &= \Sq^4(\sum_{\substack{i< j}} \prod_{r\neq i,j} x_r) \,\,\,\,=\,\,\,\,\sum_{\substack{i< j}}( \prod_{r\neq i,j} x_r)\cdot (\sum_{\substack{k<l\\ \{k,l\} \cap \{i,j\} = \emptyset}} x_kx_l) \\
 &= \sigma_n\sigma_2 - \left( n(\sum_{i\neq j} (x_i^2) \prod_{k\neq i,j} x_k ) + \binom{n+2}{2}x_1x_2\cdots x_{n+2} \right) \\
 &= \sigma_n\sigma_2 - \binom{n+2}{2}\sigma_{n+2} - n(\sigma_1\sigma_{n+1} - (n+2)\sigma_{n+2}) \\
 &= \sigma_n\sigma_2 - n\sigma_1\sigma_{n+1} + \left(-\binom{n+2}{2} + n(n+2) \right)\sigma_{n+2} \\
 &= \sigma_n\sigma_2 + n\sigma_1\sigma_{n+1} + \frac{(n-1)(n+2)}{2} \sigma_{n+2},
\end{align*}
and hence $
\Sq^4c_n = c_2c_n + n c_1c_{n+1} + \frac{(n+2)(n-1)}{2} c_{n+2}.$
This proves \Cref{eq:sq4}. The proof of \Cref{eq:sq2} is similar but easier, by considering $BU(1)^{\times (n+1)} \to BU(n+1)$ instead.

\begin{align*}
\Sq^2\sigma_n(x_1, \cdots, x_{n+1}) &= \Sq^2(\sum_{i} \prod_{r\neq i} x_r) \,\,\,= \,\,\,\sum_{i}( \prod_{r\neq i} x_r)\cdot (\sum_{\substack{ k\neq i}} x_k) \\
&=\left(\sum_{i}( \prod_{r\neq i} x_r)\right) \cdot (\sum_{\substack{ k}} x_k)-(n+1)\prod_jx_j \\
 &= \sigma_n\sigma_1 - (n+1)\sigma_{n+1}.
\end{align*}
\end{proof}


\bibliographystyle{abbrv}
\bibliography{corank2}

@article{Arone02,
author={G. Arone},
title={The {W}eiss derivatives of {BO(-) and BU(-)}},
journal={Topology},
volume={41},
year={2002},
number={3},
pages={451--481}
}

@incollection{AHSSeq,
  author    = {Atiyah, M. F. and Hirzebruch, F.},
  title     = {Vector Bundles and Homogeneous Spaces},
  booktitle = {Differential Geometry},
  series    = {Proceedings of Symposia in Pure Mathematics},
  volume    = {3},
  pages     = {7--38},
  publisher = {American Mathematical Society},
  address   = {Providence, RI},
  year      = {1961}
}

@article{AR,
author={M. Atiyah and E. Rees},
title={Vector bundles on projective 3-space}, 
journal={Inventiones Math.}, 
volume={35}, 
year={1976},
pages={131--153} 
}

@article{Bott,
 author={R. Bott},
 journal={Ann. of Math.},
 pages={313--337},
 title={The Stable Homotopy of the Classical Groups},
 volume={70},
 year={1959}
}

@article{Hu,
author={Y. Hu},
title={Metastable complex vector bundles over complex projective spaces},
journal={Trans. Amer. Math. Soc.},
volume={376},
year={2023},
number={11},
pages={7783--7814}
}

@book{MP,
author={J. P. May and K. Ponto},
title={More Concise Algebraic Topology: Localization, Completion, and Model Categories},
series={Chicago Lectures in Mathematics},
publisher={The University of Chicago Press},
address={Chicago and London},
year={2012}
}

@inproceedings{Mimura_HBAT,
author={M. Mimura},
editor={I. M. James},
booktitle={Handbook of Algebraic Topology},
title={Homotopy theory of {L}ie groups},
year={1995},
volume={58},
pages={951--991},
publisher={Elsevier}
}

@article{Opie-r3p5,
title = {A classification of complex rank 3 vector bundles on ${CP}^5$},
journal = {Advances in Mathematics},
volume = {455},
pages = {109878},
year = {2024},
issn = {0001-8708},
doi = {https://doi.org/10.1016/j.aim.2024.109878},
url = {https://www.sciencedirect.com/science/article/pii/S0001870824003931},
author = {M. Opie}
}

@article{Switzer,
author={R. M. Switzer},
title={Rank $2$ bundles over {$P^n$} and the e-Invariant},
journal={Indiana University Math. J.},
volume={28},
year={1979},
number={6},
pages={961--974}
}

@article{Switzer2,
author={R. M. Switzer},
title={Complex $2$-plane bundles over complex projective space},
journal={Math. Z.},
volume={168},
year={1979},
number={},
pages={275--287}
}

@article{Weiss,
author={M. Weiss},
title={Orthogonal calculus},
journal={Trans. Amer. Math. Soc.},
volume={347},
year={1995},
number={10},
pages={3743--3796}
}

@book{MT,
  title={Cohomology operations and applications in homotopy theory},
  author={Mosher, Robert E and Tangora, Martin C},
  year={2008},
  publisher={Courier Corporation}
}

@article{opie24enum,
  title={Enumerating complex rank {$ n $} vector bundles on {$\mathbb C P^{n+1}$}},
  author={Opie, Morgan P},
  journal={arXiv preprint arXiv:2410.23520},
  year={2024}
}

@article{taggart22unitary,
  title={Unitary calculus: model categories and convergence},
  author={Taggart, Niall},
  journal={Journal of Homotopy and Related Structures},
  volume={17},
  pages={419--462},
  year={2022},
  publisher={Springer Science+ Business Media}
}

@article{carrtagg24,
  title={Symplectic {W}eiss calculi},
  author={Carr, Matthew and Taggart, Niall},
  journal={arXiv preprint arXiv:2404.11796},
  year={2024}
}

@article{singer73,
  title={Steenrod squares in spectral sequences. {I}},
  author={Singer, William M},
  journal={Transactions of the American Mathematical Society},
  volume={175},
  pages={327--336},
  year={1973}
}

\end{document}